\documentclass[11pt, reqno]{amsart}

\usepackage{etoolbox}

\newbool{ForSubmission}
\booltrue{ForSubmission}

\usepackage{amscd,amsmath}
\usepackage{mathrsfs}
\usepackage{amsfonts}
\usepackage{amssymb, bm, xspace}
\usepackage{enumerate}
\usepackage{setspace}
\usepackage{comment}		

\usepackage{datetime}

\usepackage{xcolor}

\usepackage[pagebackref=true, colorlinks=true, citecolor=blue]{hyperref}

\usepackage[margin=1.2in]{geometry}

\usepackage{mathtools}

\usepackage{cleveref}

\usepackage{stackengine}
\stackMath

\newcommand{\Comment}[1]{{\color{blue}#1}}
\newcommand{\OptionalDetails}[1]{
    \ifbool{ForSubmission}
        {
        }
        {\begin{quote}\Comment{\footnotesize
        \medskip

        \noindent#1}
        \end{quote}
        }
    }

\newbool{HaveBBM}
\booltrue{HaveBBM}

\ifbool{HaveBBM}{
	\usepackage{bbm}
    }
    {
    }

\newbool{arXivFormat}
\boolfalse{arXivFormat}
    
\newcommand{\IfarXivElse}[2]{
    \ifbool{arXivFormat}
        {#1}{#2}
    }

\renewcommand{\mathbf}[1]{\bm{#1} \textbf{ *** Use bm instead of mathbf ***}}

\newcommand{\eqn}{\begin{eqnarray}}
\newcommand{\een}{\end{eqnarray}}

\newtheorem{theorem}{Theorem}[section]
\newtheorem*{theorem*}{Theorem}				
\newtheorem{prop}[theorem]{Proposition}
\newtheorem{lemma}[theorem]{Lemma}
\newtheorem{cor}[theorem]{Corollary}

\newtheorem{definition}[theorem]{Definition}
\newtheorem{remark}[theorem]{Remark}

\newtheorem*{assumptions*}{Assumptions}
\numberwithin{equation}{section}

\newcommand{\abs}[1]{\left\vert#1\right\vert}

\newcommand{\innp}[1]{\ensuremath{\left< #1 \right>}}

\newcommand{\BB}[1]{\ensuremath{\mathbb{#1}}}

\newcommand{\R}{\ensuremath{\BB{R}}}
\newcommand{\N}{\ensuremath{\BB{N}}}
\newcommand{\Z}{\ensuremath{\BB{Z}}}
\newcommand{\iny}{\ensuremath{\infty}}
\newcommand{\grad}{\ensuremath{\nabla}}

\newcommand{\CharFunc}{
    \ifbool{HaveBBM}{
        \ensuremath{\mathbbm{1}}
        }
        {
        \ensuremath{\bm{1}}
        }
    }

\DeclareMathOperator{\dv}{div} %
\DeclareMathOperator{\curl}{curl} %
\DeclareMathOperator{\supp}{supp} %
\DeclareMathOperator{\erf}{erf} %
\DeclareMathOperator{\erfc}{erfc} %
\newcommand{\prt}{\ensuremath{\partial}}
\newcommand{\brac}[1]{\ensuremath{\left[ #1 \right]}}

\newcommand{\pr}[1]{\ensuremath{\left( #1 \right)}}

\DeclarePairedDelimiter{\smallpr}{\big(}{\big)}
\DeclarePairedDelimiter{\set}{\{}{\}}

\newcommand{\bigset}[1]{\ensuremath{\left\{ #1 \right\}}}

\newcommand{\norm}[1]{\ensuremath{\left\Vert #1 \right\Vert}}

\newcommand{\smallnorm}[1]{\ensuremath{\Vert #1 \Vert}}
\newcommand{\mednorm}[1]{\ensuremath{\Big\Vert #1 \Big\Vert}}

\newcommand\tenq[2][1]{%
	\def\useanchorwidth{T}%
	\ifnum#1>1%
		\stackunder[0pt]{\tenq[\numexpr#1-1\relax]{#2}}{\scriptscriptstyle\sim}%
	\else%
		\stackunder[1pt]{#2}{\scriptscriptstyle\sim}%
	\fi%
	}

\newcommand{\wh}{\widehat}

\renewcommand{\epsilon}{\varepsilon}

\newcommand{\Cal}[1]{\ensuremath{\mathcal{#1}}}
\newcommand{\al}{\ensuremath{\alpha}}
\newcommand{\la}{\ensuremath{\lambda}}

\newcommand{\ol}{\overline}

\newcommand{\smallabs}[1]{\ensuremath{\vert #1 \vert}}
\newcommand{\medabs}[1]{\ensuremath{\big\vert #1 \big\vert}}

\newcommand{\Holder}
    {H\"{o}lder\xspace}

\newcommand{\Ignore}[1]{}

\newcommand{\ToDo}[1]{\textbf{\Comment{[#1]}}}

\definecolor{Correction}{named}{red}

\crefname{cor}{Corollary}{Corollaries} 
									   
\crefname{lemma}{Lemma}{Lemmas}	       

\crefname{section}{Section}{Sections}
\Crefname{section}{Section}{Sections}

\crefname{appendix}{Appendix}{Appendices}
\Crefname{appendix}{Appendix}{Appendices}

\crefname{theorem}{Theorem}{Theorems}
\Crefname{theorem}{Theorem}{Theorems}

\crefname{prop}{Proposition}{Propositions}
\Crefname{prop}{Proposition}{Propositions}

\crefname{conj}{Conjecture}{Conjectures}
\Crefname{conj}{Conjecture}{Conjectures}

\crefname{definition}{Definition}{Definitions}
\Crefname{definition}{Definition}{Definitions}

\crefname{remark}{Remark}{Remarks}
\Crefname{remark}{Remark}{Remarks}

\crefname{assumptions}{Assumptions}{Assumptions}
\Crefname{assumptions}{Assumptions}{Assumptions}

\crefformat{equation}{(#2#1#3)}
\crefrangeformat{equation}{(#3#1#4) through (#5#2#6)}
\crefmultiformat{equation}
    {(#2#1#3)}%
    { and~(#2#1#3)}
    {, (#2#1#3)}
    { and~(#2#1#3)}

\newcommand{\stardot}{\mathop{* \cdot}}

\newcommand{\MOC}{modulus of continuity\xspace}
\newcommand{\CapMOC}{Modulus of continuity\xspace}

\newcommand{\ProofStep}[1]{\bigskip\noindent\textbf{#1}}
\newcommand{\Case}[2]{\smallskip\noindent\textbf{Case #1} (#2):}

\newcommand{\DontSubmit}[1]{
    \ifbool{ForSubmission}
        {
        }
        {\begin{quote}\Comment{\footnotesize
        \medskip

        \noindent#1}
        \end{quote}
        }
    }

\newcommand{\omu}{\ol{\mu}}

\newcommand{\subadd}{subadditive\xspace} 

\newcommand{\GB}{growth bound\xspace}
\newcommand{\GBs}{growth bounds\xspace}
\newcommand{\GBU}{well-posedness growth bound\xspace}
\newcommand{\GBUs}{well-posedness growth bounds\xspace}
\newcommand{\GBE}{global well-posedness growth bound\xspace}

\newcommand{\GBany}{(well-posedness, global well-posedness) growth bound\xspace}

\newcommand{\preGB}{pre-growth bound\xspace}
\newcommand{\preGBs}{pre-growth bounds\xspace}

\newcommand{\BVBV}{bounded vorticity, bounded velocity\xspace}

\begin{document}
\newdateformat{mydate}{\THEDAY~\monthname~\THEYEAR}

\title
	[2D Euler when velocity grows at infinity]
	{Well-posedness of the 2D Euler equations when velocity grows at infinity}

\author{Elaine Cozzi}
\address{Department of Mathematics, 368 Kidder Hall, Oregon State University, Corvallis, OR 97330, U.S.A.}
\email{cozzie@math.oregonstate.edu}

\author{James P Kelliher}
\address{Department of Mathematics, University of California, Riverside, USA}
\email{kelliher@math.ucr.edu}	
	






\date{\today}

\begin{abstract}
We prove the uniqueness and finite-time existence of bounded-vorticity solutions to the 2D Euler equations having velocity growing slower than the square root of the distance from the origin, obtaining global existence for more slowly growing velocity fields. We also establish continuous dependence on initial data.
\end{abstract}

\maketitle


		


\tableofcontents
	

\section{Introduction}\label{S:Introduction}

\noindent
In \cite{Serfati1995A}, Ph. Serfati proved the existence and uniqueness of Lagrangian solutions to the 2D Euler equations having bounded vorticity and bounded velocity (for a rigorous proof, see \cite{AKLN2015}). Our goal here is to discover how rapidly the velocity at infinity can grow (keeping the vorticity bounded) and still obtain existence or uniqueness of solutions to the 2D Euler equations.

Serfati's approach in \cite{Serfati1995A} centered around a novel identity he showed held for bounded vorticity, bounded velocity solutions. We write this identity, which we call \textit{Serfati's identity} or the \textit{Serfati identity}, in the form
\begin{align}\label{e:SerfatiId}
	\begin{split}
		u^j(t&) - (u^0)^j
			= (a_\la K^j) *(\omega(t) - \omega^0)
				- \int_0^t \pr{\grad \grad^\perp \brac{(1 - a_\la) K^j}}
					\stardot (u \otimes u)(s) \, ds,
		\end{split}
\end{align}
$j = 1, 2$.
Here, $u$ is a divergence-free velocity field, with $\omega = \curl u := \prt_1 u^2 - \prt_2 u^1$ its vorticity, and
$
	K(x) := (2 \pi)^{-1} x^\perp \abs{x}^{-2}
$
is the Biot-Savart kernel,
$x^\perp := (-x_2, x_1)$. The function $a_\la$, $\la > 0$, is a scaled radial cutoff function (see \cref{D:Radial}).
Also, we have used the notation,
\begin{align*}
    \begin{array}{ll}
    v \stardot w
 		= v^i * w^i
 	        &\mbox{if $v$ and $w$ are  vector fields}, \\
    A \stardot B
		= A^{ij} * B^{ij}
		    &\mbox{if $A$, $B$ are matrix-valued functions on $\R^2$},
    \end{array}
\end{align*}
where $*$ denotes convolution and where repeated indices imply summation.

The Biot-Savart law, $u = K * \omega$, recovers the unique divergence-free vector field $u$ decaying at infinity from its vorticity (scalar curl) $\omega$. It does not apply to \BVBV solutions---indeed this is the greatest
difficulty to overcome with such solutions---but because of the manner in which $a_\la$ cuts off the Biot-Savart kernel in \cref{e:SerfatiId}, we see that all the terms in Serfati's identity are finite. In fact, there is room for growth at infinity both in the vorticity and the velocity, though to avoid excessive complications, we only treat growth in the velocity. Because $\grad \grad^\perp \brac{(1 - a_\la) K^j(x)}$ decays like $\smallabs{x}^{-3}$, we see that as long as $\abs{u(x)}$ grows more slowly than $\smallabs{x}^{1/2}$, all the terms in \cref{e:SerfatiId} will at least be finite.

In brief, we will establish uniqueness of solutions having $o(\smallabs{x}^{1/2})$ growth along with  finite-time existence. We will obtain global existence only for much more slowly growing velocity fields.

To give a precise statement of our results, we must first describe the manner in which we prescribe the growth of the velocity field at infinity and give our formulation of a weak solution, including the function spaces in which the solution is to lie.

In what follows, an \textit{increasing function} means nondecreasing; that is, not necessarily \textit{strictly} increasing. 

We define four types of increasingly restrictive bounds on the growth of the velocity, as follows:
\begin{definition}\label{D:GrowthBound}[Growth bounds]
	\begin{enumerate}[(i)]
	\item
	A \textit{\preGB} is
	a function $h \colon [0, \iny) \to (0, \iny)$
	that is concave,
	increasing,
	differentiable on $[0, \iny)$,
	and twice continuously differentiable on $(0, \iny)$.
	
	\item
	A \textit{\GB} $h$ is a \preGB for which 
			$ 
				\int_1^\iny h(s) s^{-2} \, ds < \iny
			$.
	
	\item
	A \textit{\GBU} $h$ is a \GB for which $h^2$ is also a \GB.
	
	\item
	Let $h$ be a \GBU and define
	$H[h] \colon (0, \iny) \to (0, \iny)$ by
	\begin{align}\label{e:HDef}
		H[h](r)
			:= \int_r^\iny \frac{h(s)}{s^2} \, ds,
	\end{align}
	noting that the condition in $(ii)$ insures $H[h]$ and $H[h^2]$
	are well-defined.	
	We show in \cref{L:EProp} that there always exists a
	continuous, convex function $\mu$ with $\mu(0) = 0$
	for which
	\begin{align*}
		\displaystyle
		E(r)
			:= \pr{1 + r^{\frac{1}{2}} H[h^2] \pr{r^{\frac{1}{2}}}}^2 r
			\le \mu(r).
	\end{align*}
	We call $h$ a \GBE if for some such $\mu$,
	\begin{align}\label{e:omuOsgoodAtInfinity}
		\int_1^\iny
			\frac{dr}
			{\mu(r)}
			= \iny.
	\end{align}
	\end{enumerate}
\end{definition}

\begin{definition}\label{D:SerfatiVelocity}
	Let $h$ be a \GB.
	We denote by $S_h$ the Banach space of all divergence-free vector
	fields, $u$, on $\R^2$ for which $u/h$ and $\omega(u)$ are bounded,
	where $\omega(u) := \prt_1 u^2 - \prt_2 u^1$ is the vorticity.
	We endow $S_h$ with the norm
	\begin{align*}
		\norm{u}_{S_h}
			:= \norm{u/h}_{L^\iny} + \norm{\omega(u)}_{L^\iny}.
	\end{align*}
\end{definition}

\begin{definition}[Radial cutoff function]\label{D:Radial}
	Let $a$ be a radially symmetric function in $C_c^\iny(\R^2)$
	taking values in $[0, 1]$, supported in $\ol{B_1(0)}$,
	with $a \equiv 1$ on $\ol{B_{1/2}(0)}$.
	(By $B_r(x)$ we mean the open ball of radius $r$ centered at $x$.)
	We call any such function a \textit{radial cutoff function.}
	For any $\la > 0$ define
	\begin{align*}
		a_\lambda (x)
			:= a \pr{\frac{x}{\lambda}}.
	\end{align*}
\end{definition}

\begin{definition}[Lagrangian solution]\label{D:ESol}
	Fix $T > 0$ and let $h$ be a \GB. Assume that
	$u \in C(0, T; S_h)$,
	let $\omega = \curl u := \prt_1 u^2 - \prt_2 u^1$,
	and let $X$ be the unique flow map for $u$.
	We say that $u$ is a  solution to the Euler equations
	in $S_h$
	without forcing and with initial velocity $u^0 = u|_{t = 0}$ in $S_h$
	if the following conditions hold:
	\begin{enumerate}
		\item
			$
				\omega = \omega^0 \circ X^{-1}
			$
			on $[0, T] \times \R^2$,
			where $\omega^0 := \curl u^0$;

			\smallskip
			
		\item
			Serfati's identity \cref{e:SerfatiId} holds for all $\la > 0$.
	\end{enumerate}
\end{definition}

	The existence and uniqueness of the flow map $X$ in \cref{D:ESol} is assured by
	$u \in C([0, T]; S_h)$ (see \cref{L:FlowWellPosed}). It also then follows
	easily that
	the vorticity equation, $\prt_t \omega + u \cdot \grad \omega = 0$,
	holds in the sense of distributions---so $u$ is also a weak
	Eulerian solution. (See also the discussion of Eulerian
	versus Lagrangian solutions in \cref{S:WeakFormulation}.)

Our main results are \cref{T:Existence} through \cref{T:aT}.


\begin{theorem}\label{T:Existence}
	Let $h$ be any \GBU.
	For any $u^0 \in S_h$ there is a $T > 0$
	such that there exists a
	solution to the Euler equations in $S_h$ on $[0, T]$
	as in \cref{D:ESol}. If $h$ is a \GBE
	then $T$ can be chosen to be arbitrarily large.
\end{theorem}

\begin{theorem}\label{T:ProtoUniqueness}
	Let $h$ be any \GBU and let $\zeta \ge h$ be any \GB for which
	$\zeta/h$ and $\zeta h$ are also \GBs. 
	Let $u_1^0, u_2^0 \in S_h$ and let $\omega^0_1$, $\omega^0_2$
	be the corresponding
	initial vorticities.
	Assume that there exist solutions,
	$u_1$, $u_2$, to the Euler equations in $S_h$ on $[0, T]$
	with initial velocities $u_1^0$, $u_2^0$, and let $\omega_1, \omega_2$
	and $X_1, X_2$ be the corresponding vorticities and flow maps.
	Let
	\begin{align}\label{e:aT}
		a(T)
			&:= \smallnorm{(u_1^0 - u_2^0)/\zeta}_{L^\iny(\R^2)}
				+ \norm{J/\zeta}_{L^\iny((0, T) \times \R^2)},
	\end{align}
	where
	\begin{align}\label{e:J}
		J(t, x)
			&= \pr{(a_{h(x)} K) * (\omega_1^0 - \omega_2^0) \circ X_1^{-1}(t)}(x)
				- \pr{(a_{h(x)} K) * (\omega_1^0 - \omega_2^0)}(X_1(t, x)).
	\end{align}
	Define
	\begin{align*}
		\eta(t)
			&:= \norm{\frac{X_1(t, x) - X_2(t, x)}
					{\zeta(x)}}_{L^\iny_x(\R^2)},  \\
		L(t)
			&:= 
				\norm{\frac{u_1(t, X_1(t, x)) - u_2(t, X_2(t, x))}{\zeta(x)}}
					_{L^\iny_x(\R^2)}, \\
		M(t)
			&:= \int_0^t
				L(s) \, ds, \\
		Q(t)
			&:= \norm{(u_1(t) - u_2(t))/\zeta}_{L^\iny(\R^2)}.
	\end{align*}

	Then $\eta(t) \le M(t)$ and
	\begin{align}\label{e:LProtoBoundForUniqueness}
		\int_{ta(T)}^{M(t)} \frac{dr}{\omu(C(T) r) + r}
			\le C(T) t
	\end{align}
	for all $t \in [0, T]$, where
	\begin{align}\label{e:omu}
		\omu(r) := 
			\begin{cases}
    			-r \log r &\text{if } r \le e^{-1}, \\
				e^{-1} &\text{if } r > e^{-1}.
			\end{cases}
	\end{align}
	We have
	$M(T) \to 0$ and $\norm{Q}_{L^\iny(0, T)} \to 0$ as $a(T) \to 0$.
	Explicitly,
	\begin{align*}
		M(t)
			\le (t a(T))^{e^{-C(T)t}}
	\end{align*}
	holds until $M(t)$ increases to $C(T) e^{-1}$.
	If $ta(T) \ge C(T)e^{-1}$ then
	\begin{align*}
		M(t)
			\le C(T)t a(T).
	\end{align*}
	For all $t \ge 0$, we have
	\begin{align*}
		Q(t)
			\le \brac{a(T) + C (T) \omu(C(T) M(t))} e^{C(T) t}.
	\end{align*}
\end{theorem}


Uniqueness is an immediate corollary of \cref{T:ProtoUniqueness}:
\begin{cor}\label{C:Uniqueness}
	Let $h$ be any \GBU.
	Then solutions to the 2D Euler equations on $[0, T]$ in $S_h$
	are unique.
\end{cor}
\begin{proof}
	Apply \cref{T:ProtoUniqueness} with
	$u_1^0 = u_2^0 = u(0)$ so that $a(T) = 0$ and set $\zeta = h$,
	noting that $\zeta/h = 1$ and $\zeta h = h^2$ are both \GBs.
\end{proof}

\cref{C:MainResult} applies \cref{T:Existence,C:Uniqueness} to two explicit \GBUs. The proof is simply a matter of verifying that the two given \GBUs actually satisfy the pertinent conditions in \cref{D:GrowthBound}. (Since the proof is very ``calculational,'' we defer it to \cref{S:Corollary}.) We note that the existence of solutions for $h_2$ was shown in \cite{Cozzi2015} using different techniques.

\begin{cor}\label{C:MainResult}
	Let $h_1(r) := (1 + r)^\al$ for some $\al \in [0, 1/2)$,
	$h_2(r) := \log^{\frac{1}{4}}(e + r)$.
	If $u^0 \in S_{h_j}$, $j = 1$ or $2$, then there exists a
	unique solution to the Euler equations in $S_{h_j}$ on $[0, T]$
	as in \cref{D:ESol} for some $T > 0$.
	If $u^0 \in S_{h_2}$ then $T$ can be made arbitrarily large.
\end{cor}

	Loosely speaking, \cref{C:MainResult} says that solutions in $S_h$ are unique
	and exist for finite time as long as $h^2$ is sublinear, while
	global-in-time solutions exist for velocities growing
	very slowly at infinity. These slowly growing velocities are somewhat
	analogous to the ``slightly unbounded'' vorticities of Yudovich
	\cite{Y1995}, which extends the uniqueness result for bounded
	vorticities in \cite{Y1963}.

If $u_1^0 \ne u_2^0$ then $a(T)$ does not vanish. If $u_1^0, u_2^0 \in S_h$ for some \GBU $h$, and $u_1^0 - u_2^0$ is small in $S_h$, we might expect $\norm{u_1(t) - u_2(t)}_{S_h}$ to remain small at least for some time. The $S_h$ norm, however, includes the $L^\iny$ norm of $\omega_1(t) - \omega_2(t)$, with each vorticity being transported by different flow maps. Hence, we should expect $\norm{\omega_1(t) - \omega_2(t)}_{L^\iny}$ to be of the same order as $\norm{\omega_j(t)}_{L^\iny}$, $j = 1, 2$, immediately after time zero. Thus, it is too much to ask for continuous dependence on initial data in the $S_h$ norm. In this regard, the situation is the same as for the classical bounded-vorticity solutions of Yudovich \cite{Y1963}, and has nothing to do with lack of decay at infinity. The best we can hope to obtain is a bound on $(u_1(t) - u_2(t))/\zeta$ in $L^\iny$---and so, by interpolation, in $C^\al$ for all $\al < 1$.

To obtain continuous dependence on initial data or control how changes at a distance from the origin affect the solution near the origin (\textit{effect at a distance}, for short), we can employ the bound on $Q$ in \cref{T:ProtoUniqueness} to obtain a bound on how far apart the two solutions become, weighted by $\zeta$. For continuous dependence on initial data, $\zeta = h$ is most immediately pertinent; for controlling effect at a distance, $\zeta \geq h$ is better.


The simplest form of continuous dependence on initial data, which follows from \cref{T:aTSimple} applied to \cref{T:ProtoUniqueness}, shows that if the initial velocities are close in $S_\zeta$ then they remain close (in a weighted $L^\iny$ space) for some time.

\begin{theorem}\label{T:aTSimple}
	Make the assumptions in \cref{T:ProtoUniqueness}.
	Then
	\begin{align*}
		a(T)
			\le C \norm{u_1^0 - u_2^0}_{S_\zeta}.
	\end{align*}
\end{theorem}

While \cref{T:aTSimple} gives a theoretically meaningful measure of continuous dependence on initial data, the assumption that the initial velocities are close in $S_\zeta$ is overstrong. For instance, it would not apply to two vortex patches that do not quite coincide. One approach, motivated in part by this vortex patch example, is to make some assumption on the closeness of the initial vorticities locally uniformly in an $L^p$ norm for $p < \iny$, as was done in \cite{TaniuchiEtAl2010,AKLN2015}. This assumption is, however, unnecessary (and in our setting somewhat artificial) as shown for the special case of bounded velocity  ($h \equiv 1$) in \cite{CK2016}.

In \cite{CK2016}, the focus was not on continuous dependence on initial data, per se, but rather on understanding the effect at a distance. Hence, we used a function, $\zeta(r) = (1 + r)^\al$ for any $\al \in (0, 1)$, in place of $h$ and obtained a bound on $(u_1(t) - u_2(t))/\zeta$ in terms of $(u_1^0 - u_2^0)/\zeta$, each in the $L^\iny$ norm. In \cref{T:aT}, we obtain a similar bound using very different techniques. Our need to assume bounded velocity ($h$ = 1) arises from our inability to obtain usable transport estimates for non-Lipschitz vector fields growing at infinity. It is not clear whether this is only a technical issue or represents some fundamental new phenomenon causing unbounded velocities to have less stability, in the sense of effect at a distance, than bounded velocities. (See \cref{R:TechnicalIssue}.)

\begin{theorem}\label{T:aT}
	Let $u_1^0, u_2^0 \in S_1$
	and let $\zeta$ be a \GB.
	Let $u_1$, $u_2$ be the corresponding solutions in $S_1$
	on $[0, T]$
	with initial velocities $u_1^0$, $u_2^0$.
	Fix $\delta$, $\al$ with $0 < \delta < \al < 1$
	and choose any $T^* > 0$ such that
	\begin{align}\label{e:TStar}
		T^*
			< \min\bigset{T,
				 \frac{1 + \delta}{C C_0}
			},
	\end{align}
	where
	\begin{align}\label{e:C0C1}
		C_0 = \smallnorm{u_1}_{L^\iny(0, T; S_1)}.
	\end{align}
	Then
	\begin{align}\label{e:aTstarBound}
		a(T^*)
			\le C_1 \Phi_\al \pr{
					T,
					C
					\norm{\frac{u_1^0 - u_2^0}{\zeta}}_{L^\iny}^\delta
					},
	\end{align}
	where,
	\begin{align}\label{e:Phialpha}
		\Phi_\al(t, x)
			:= x + x^{\frac{e^{-C_0t}}{\al + e^{-C_0t}}}.
	\end{align}
\end{theorem}

\Ignore{  
\begin{theorem}\label{T:aT}
	Let $u_1^0, u_2^0 \in S_1$
	and let $\zeta$ be a \GB.
	Let $u_1$, $u_2$ be the corresponding solutions in $S_1$
	on $[0, T]$
	with initial velocities $u_1^0$, $u_2^0$.
	Fix $\delta$, $\al$ with $0 < \delta < \al < 1$
	and choose any $T^* > 0$ such that
	\begin{align}\label{e:TStar}
		T^*
			< \min\bigset{T, \frac{1}{C_0}
				\log \pr{\frac{2 + \al}{2+\delta}}}.
	\end{align}
	Then
	\begin{align}\label{e:aTstarBound}
		a(T^*)
			\le C_1 \Phi_\al \pr{
					T,
					C_{\delta, \al}
					\norm{\frac{u_1^0 - u_2^0}{\zeta}}_{L^\iny}^\delta
					},
	\end{align}
	where,
	\begin{align}\label{e:C0C1}
		C_0 = \smallnorm{u_1}_{L^\iny(0, T; S_1)}, \quad
		C_{\delta, \al}
			 = \frac{C(T, \zeta)}{\delta (1 - \al)}
			\brac{\smallnorm{u_1^0}_{S_1} + \smallnorm{u_2^0}_{S_1} + 1},
	\end{align}
	and
	\begin{align}\label{e:Phialpha}
		\Phi_\al(t, x)
			:= x + x^{\frac{e^{-C_0t}}{\al + e^{-C_0t}}}.
	\end{align}
\end{theorem}  } 

\begin{remark}\label{R:aT}
	To obtain a bound on $a(T)$, we iterate
	the bound in \cref{e:aTstarBound}
	$N$ times, where $T = N T^*$ (decreasing $T^*$ if necessary
	so as to obtain
	the minimum possible positive integer $N$),
	applying the bound on $Q(t)$ from
	\cref{T:ProtoUniqueness} after each step.
	Since $T^*$ depends only upon $C_0$, which does not change,
	we can always iterate this way.
	In principle, the resulting bound can be made explicit, at least
	for sufficiently small $a(T)$.
\end{remark}

The bound on $u_1 - u_2$ given by the combination of \cref{T:ProtoUniqueness,T:aT} is not optimal, primarily because we could use an $L^1$-in-time bound on $J(t, x)$ in place of the $L^\iny$-in-time bound. (This would also be reflected in the bound on $a(T^*)$ in \cref{e:aTstarBound}.) This would not improve the bounds sufficiently, however, to justify the considerable complications to the proofs.

The issue of well-posedness of solutions to the 2D Euler equations with bounded vorticity but velocities growing at infinity was taken up recently by Elgindi and Jeong in \cite{ElgindiJeong2016}. They prove existence and uniqueness of such solutions for velocity fields growing \textit{linearly} at infinity under the assumption that the vorticity has $m$-fold symmetry for $m \ge 3$. We study here solutions with no preferred symmetry, and our approach is very different; nonetheless, aspects of our uniqueness argument were influenced by Elgindi's and Jeong's work. In particular, the manner in which they first obtain elementary but useful bounds on the flow map inspired \cref{L:FlowBounds}, and the bound in \cref{P:Morrey} is the analog of Lemma 2.8 of \cite{ElgindiJeong2016} (obtained differently under different assumptions).

Finally, we remark that it would be natural to combine the approach in \cite{ElgindiJeong2016} with our approach here to address the case of two-fold symmetric vorticities ($m = 2$). The goal would be to obtain Elgindi's and Jeong's result, but for velocities growing infinitesimally less than linearly at infinity. A similar argument might also work for solutions to the 2D Euler equations in a half plane having sublinear growth at infinity.

\bigskip

This paper is organized as follows: We first make a few comments on our formulation of a weak solution in \cref{S:WeakFormulation}. \crefrange{S:GBs}{S:FlowMap} contain preliminary material establishing useful properties of \GBs, estimates on the Biot-Savart law and locally log-Lipschitz velocity fields, and bounds on flow maps for velocity fields growing at infinity. 
We establish the existence of solutions, \cref{T:Existence}, in \cref{S:Existence}.
We prove \cref{T:ProtoUniqueness} in \cref{S:Uniqueness}, thereby establishing the uniqueness of solutions, \cref{C:Uniqueness}. In \cref{S:ContDep}, we prove \cref{T:aTSimple,T:aT}, establishing continuous dependence on initial data and controlling the effect of changes at a distance.

In \cref{S:LP}, we employ Littlewood-Paley theory to establish estimates in negative \Holder spaces for $u \in L^\iny(0, T; S_1)$.  These estimates are used in the proof of \cref{T:aT} in \cref{S:ContDep}.

 
 Finally, \cref{S:Corollary} contains the proof of \cref{C:MainResult}.

%
%
\section{Some comments on our weak formulation}\label{S:WeakFormulation}
\noindent In our formulation of a weak solution in \cref{D:ESol} we made the assumption that Serfati's identity holds. Yet formally, Serfati's identity and the Euler equations are equivalent, and with only mild assumptions on regularity and behavior at infinity, they are rigorously equivalent. We do not purse this equivalence in any detail here, as it would be a lengthy distraction, we merely outline the formal argument.

That a solution to the Euler equations must satisfy Serfati's identity for smooth, compactly supported solutions (and hence also formally) is shown in Proposition 4.1 of \cite{AKLN2015}. 
\Ignore{ 
An easy way to show the reverse implication is to differentiate both sides of \cref{e:SerfatiId}, an identity that holds for all $\la > 0$, with respect to $\la$.
This gives
\begin{align*}
	\varphi_\la^j * (\omega(t) - \omega^0)
		+ \int_0^t (\grad \grad^\perp \varphi_\la^j) \stardot (u \otimes u)(s)
				\, ds = 0,
\end{align*}
where
\begin{align*}
	\varphi_\la^j(x)
		&= \prt_\la a_\la(x) K^j(x)
		= - \brac{\grad a \pr{\frac{x}{\la}}
			\cdot \pr{\frac{x}{\la^2}}} K^j(x)
		= -\brac{\grad a \pr{\frac{x}{\la}}
			\cdot \pr{\frac{x}{\la^2}}} \pr{\frac{1}{\la}}
				K^j \pr{\frac{x}{\la}} \\
		&= - \frac{1}{\la^2} \brac{\grad a \pr{\frac{x}{\la}}
			\cdot \pr{\frac{x}{\la}}} K^j \pr{\frac{x}{\la}}.
\end{align*}
We see that
\begin{align*}
	\varphi_1^j(x)
		= \varphi^j(x) := -(\grad a(x) \cdot x) K^j(x), \quad
		\varphi_\la^j(x)
			= \la^{-2} \varphi^j(\la^{-1} x).
\end{align*}
But, $\grad a(x) \cdot x$ is radially symmetric and $K^j(x)$ integrates to zero over circles. Hence, $\varphi_\la^j$ is an approximation to the identity, but it has total mass zero. Hence, it is useless.
\begin{align*}
	\int_{\R^2} \varphi_j
		= \int_{\R^2} a(x) \dv (K^j(x) x)
\end{align*}
} 
The reverse implication is more difficult. One approach is to start with the velocity identity,
\begin{align}\label{e:CKId}
	\begin{split}
		\varphi_\la * (u^j(t) - (u^0)^j)
            	= - \int_0^t \pr{\grad \grad^\perp \brac{(1 - a_\la) K^j}}
					\stardot (u \otimes u)(s) \, ds,
	\end{split}
\end{align}
where $\varphi_\la := \grad a_\la \cdot K^\perp \in C_c^\iny(\R^2)$. This identity can be shown to be equivalent to the Serfati identity by exploiting Lemma 4.4 of \cite{KBounded2015}. It turns out that  $\varphi_\la *$ is a mollifier: taking $\la \to 0$ yields (after a long calculation) the velocity equation for the Euler equations in the limit.

Assuming Serfati's identity holds, then, introduces some degree of redundancy in \cref{D:ESol}, but this redundancy cannot be entirely eliminated if we wish to have uniqueness of solutions. This is demonstrated in \cite{KBounded2015}, where it is shown that for bounded vorticity, bounded velocity solutions, Serfati's identity must hold---up to the addition of a time-varying, constant-in-space vector field. This vector field, then, serves as the uniqueness criterion; its vanishing is equivalent to the sublinear growth of the pressure (used as the uniqueness criterion in \cite{TaniuchiEtAl2010}).

Bounded vorticity, bounded velocity solutions are a special case of the solutions we consider here, but the technology developed in \cite{KBounded2015} does not easily extend to velocities growing at infinity. Hence, we are unable to dispense with our assumption that the Serfati identity holds, as we need it as our uniqueness criterion.

Another closely related issue is that we are using Lagrangian solutions to the Euler equations, as we need to use the flow map in our uniqueness argument (with Serfati's identity). We note that it is not sufficient to simply know that the vorticity equation, $\prt_t \omega + u \cdot \grad \omega = 0$, is satisfied. This tells us that the vorticity is transported, in a weak sense, by the unique flow map $X$, but we need that this weak transport equation has $\omega^0(X^{-1}(t, x))$ as its unique solution. Only then can we conclude that the curl of $u$ truly is $\omega^0(X^{-1}(t, x))$.  (For bounded velocity, bounded vorticity solutions, the transport estimate from \cite{BahouriCheminDanchin2011} in \cref{P:f0XBound} would be enough to obtain this uniqueness.)


The usual way to establish well-posedness of Eulerian solutions to the 2D Euler equations is to construct Lagrangian solutions (which are automatically Eulerian) and then prove uniqueness using the Eulerian formulation only. Such an approach works for bounded vorticity, bounded velocity solutions, as uniqueness using the Eulerian formulation was shown in \cite{TaniuchiEtAl2010} (see also \cite{CK2016}). Whether this can be extended to the solutions we study here is a subject for future work.

%
%
\section{Properties of \GBs}\label{S:GBs}
 
\noindent We establish in this section a number of properties of \GBs.

\begin{definition}\label{D:Subadd}
	We say that $f \colon [0, \iny) \to [0, \iny)$ is \textit{\subadd}
	if
	\begin{align}\label{e:subadd}
		f(r + s) \le
			f(r) + f(s)
			\text{ for all } r, s \ge 0.
	\end{align}
\end{definition}

\begin{lemma}\label{L:hGood}
	Let $h$ be a \preGB. Then $h$ is subadditive, as is $h^2$ if $h$ is a \GBU.
	Also, $h(r) \le c r + d$ with $c = h'(0)$, $d = h(0)$,
	and the analogous statement holds for $h^2$
	when $h$ is a \GBU.
\end{lemma}
\begin{proof}
	Because $h$ is concave with $h(0) \ge 0$, we have
	\begin{align*}
		a h(x)
			\le a h(x) + (1 - a) h(0)
			\le h(ax + (1 - a) 0) = h(ax)
	\end{align*}
	for all $x > 0$ and $a\in[0,1]$. We apply this twice with $x = r + s > 0$
	and $a = r/(r + s)$, giving
	\begin{align*}
		h(r + s)
			&= a h(r + s) + (1 - a) h(r + s)
			\le h \pr{a(r + s)} + h \pr{(1 - a) (r + s)}
			= h(r) + h(s).
	\end{align*}
	Because $h'(0) < \iny$ and $h$ is concave
	we also have $h(r) \le c r + d$.
	The facts regarding $h^2$ follow in the same way.
\end{proof}

\cref{L:EProp} gives the existence of the function $\mu$ promised in \cref{D:GrowthBound}. A $\mu$ that yields a tighter bound on $E \le \mu$ will result in a longer existence time estimate for solutions, as we can see from the application of \cref{L:NotOsgood} in the proof of \cref{T:Existence}. The estimate we give in \cref{L:EProp} is very loose; in specific cases, this bound can be much improved.
\begin{lemma}\label{L:EProp}
	Let $h$ be a \GBU. There exists a
	continuous, convex function $\mu$ with $\mu(0) = 0$
	for which $E \le \mu$, where $E$ is as in \cref{D:GrowthBound}.
\end{lemma}
\begin{proof}
	Since $h(0) > 0$, $H(r) := H[h^2](r) \to \iny$
	as $r \to 0^+$. Hence, L'Hospital's rule gives
	\begin{align*}
		\lim_{r \to 0^+} r H(r)
			&= \lim_{r \to 0^+} \frac{H(r)}{r^{-1}}
			= -\lim_{r \to 0^+} \frac{H'(r)}{r^{-2}}
			= \lim_{r \to 0^+} \frac{h^2(r) r^{-2}}{r^{-2}}
			= h^2(0).
	\end{align*}
	It
	follows that $r^{\frac{1}{2}} H(r^{\frac{1}{2}}) \le C$ for $r \in (0, 1)$.
	Then, since $H$ decreases, we see that
	$(1 + r^{\frac{1}{2}} H(r^{\frac{1}{2}}))^2
	\le (1 + C r^{\frac{1}{2}})^2
	\le 2(1 + Cr) \le C(1 + r)$.
	Hence, we can use $\mu(r) = Cr (1 + r)$.
\end{proof}

\Ignore { 
\begin{lemma}\label{L:EquivOsgood}
	The condition on $h$ in \cref{e:omuOsgoodAtInfinity} is equivalent to
	\begin{align*}
		I
			:= \int_1^\iny \frac{dr}
				{\pr{1 + r^{\frac{1}{2}} H \pr{r^{\frac{1}{2}}}}^2 r}
			= \iny.
	\end{align*}
\end{lemma}
\begin{proof}
	Making the change of variables, $r = s^2$, so that
	$dr/r = 2 s \, ds/r = 2 s \, ds/s^2 = 2 \, ds/s$, gives
	\begin{align*}
		I
			:= 2 \int_1^\iny \frac{ds}
				{\pr{1 + s H(s)}^2 s}.
	\end{align*}
\end{proof}
} 

\begin{remark}\label{R:hSubadditive}
	Abusing notation, we will often write $h(x)$ for $h(\abs{x})$,
	treating $h$ as a map from $\R^2$ to $\R$. Treated this way,
	$h$ remains subadditive in the sense that
	\begin{align*}
		h(y)
			= h(\abs{y})
			\le h(\abs{x - y} + \abs{x})
			\le h(\abs{x - y})+ h(\abs{x})
			= h(x - y) + h(x).
	\end{align*}
	Here, we used
	the triangle inequality and that $h \colon [0, \iny) \to (0, \iny)$
	is increasing and subadditive.
	Similarly,
	$
		\abs{h(y) - h(x)} \le h(x - y).
	$
\end{remark}

\begin{lemma}\label{L:preGB}
	Let $h$ be a \preGB as in \cref{D:GrowthBound}.
	Then for all $a \ge 1$ and $r \ge 0$,
	\begin{align}\label{e:flar}
		&h(a r)
			\le 2 a h(r)
	\end{align}
	and
	\begin{align}\label{e:flarflar}
		\begin{split}
		&h(a h(r))
			\le C(h) a h(r), \\
		&h(h(r))/h(r) \le C(h)
		\end{split}
	\end{align}
	where $C(h) = 2 (h'(0) + h(h(0))/h(0)$).
\end{lemma}
\begin{proof}
	To prove \cref{e:flar}, we will show first that
	for any positive integer $n$ and any $r \ge 0$,
	\begin{align}\label{e:fn}
		h(n r)
			\le n h(r).
	\end{align}
	For $n = 1$, \cref{e:fn} trivially holds,
	so assume that \cref{e:fn} holds for $n - 1 \ge 1$.
	Then because $h$ is subadditive (\cref{L:hGood}),
	\begin{align*}
		h(n r)
			&= h((n - 1) r + r)
			\le h((n - 1)r) + h(r)
			\le (n - 1) h(r) + h(r)
			= n h(r).
	\end{align*}
	Thus, \cref{e:fn} follows for all positive integers
	$n$ by induction.
	
	If $a = n + \al$ for
	some $\al \in [0, 1)$ then
	\begin{align*}
		h(a r)
			&= h(nr + \al r)
			\le h(nr) + h(\al r)
			\le n h(r) + h(r)
			= (n + 1) h(r) \\
			&= \frac{n + 1}{n + \al} (n + \al) h(r)
			= \frac{n + 1}{n + \al} a h(r)
			\le \sup_{n \ge 1} \brac{\frac{n + 1}{n + \al}} a h(r)
			\leq 2 a h(r).
	\end{align*}
	(Note that the supremum is over $n \ge 1$ since we assumed
	that $a \ge 1$.)
	
	For \cref{e:flarflar}$_1$, let $c = h'(0)$, $d = h(0)$ as in
	\cref{L:hGood}, so that $h(r) \le cr + d$. Then,
	\begin{align*}
		h(a h(r))
			&\le h(a c r + a d)
			\le h(a c r) + h(a d)
			\le 2a (c h(r) + h(d))
			\le 2\pr{c + \frac{h(d)}{h(0)}} a h(r),
	\end{align*}
	which is \cref{e:flarflar}$_1$. From this,
	\cref{e:flarflar}$_2$ follows immediately.
\end{proof}

In \cref{S:Uniqueness} we will employ the functions, $\Gamma_t, F_t \colon [0, \iny) \to (0, \iny)$, defined for any $t \in [0, T]$ in terms of an arbitrary \GB $h$ by
	\begin{align}\label{e:GammaDef}
		\int_{a}^{\Gamma_t(a)} \frac{dr}{h(r)}
			= C t,
	\end{align}
	\begin{align}\label{e:FDef}
		F_t(a)
			= F_t[h](a)
			:= h(\Gamma_t(a)).
	\end{align}
	We know that
	$\Gamma_t$ and so $F_t$ are well-defined, because
	\begin{align}\label{e:hOsgoodAtInfinity}
		\int_1^\iny \frac{dr}{h(r)}
			\ge \int_1^\iny \frac{1}{c r + d} \, dr
			= \iny,
	\end{align}
	recalling that $h(r) \le c r + d$ by \cref{L:hGood}.
	
\begin{remark}
	If $h(0)$ were zero, then $\Gamma_t$ would be the bound at time $t$ on the
	spatial \MOC of the flow map for a velocity field having $h$ as its \MOC.
	Much is known about properties of $\Gamma_t$
	(they are explored at length in \cite{KFlow}), and most of these properties
	are unaffected by $h(0)$ being positive. One key difference, however,
	is that $\Gamma_t(0) > 0$ and $\Gamma_t'(0) < \iny$.
	As we will see in \cref{L:FProperties},
	this implies that $\Gamma_t$ is \subadd.
	This is in contrast to what happens when $h(0) = 0$, where
	$\Gamma_t(0) = 0$, $\Gamma_t'(0) = \iny$, and $\Gamma$
	satisfies the Osgood condition.
\end{remark}

\cref{L:FProperties} shows that $F_t$ is a \GB that is equivalent to $h$ in that it is bounded above and below by constant multiples of $h$.

\begin{lemma}\label{L:FProperties}
	Assume that $h$ is a \GBany and define $F_t$ as in \cref{e:FDef}.
	For all $t \in [0, T]$, $F_t$ is a \GBany
	as in \cref{D:GrowthBound}. Moreover,
	$F_t(r)$ is increasing in $t$ and $r$ with
	$h \le F_t \le C(t) h$, $C(t)$ increasing with time.
\end{lemma}
\begin{proof}
	First observe that 
	$
		\Gamma_t'(r)
			= h(\Gamma_t(r))/h(r)
	$
	follows from differentiating both sides of \cref{e:GammaDef}.
	Thus, $\Gamma_t$ is increasing and continuously differentiable
	on $(0, \iny)$. Since $\Gamma_t'(0) = h(\Gamma_t(0))/h(0) \ge 1$
	is finite,
	$\Gamma_t$ is, in fact, differentiable on $[0, \iny)$.
	Also,
	$F_t'(0) = h'(\Gamma_t(0)) \Gamma_t'(0)$ is finite and hence
	so is $(F_t^2)'(0)$, meaning that $F_t$ and $F_t^2$ are differentiable
	on all of $[0, \iny)$.
	
	We now show that $F_t$ is increasing, concave, 
	and twice differentiable on $(0, \iny)$, and that the
	same holds true for $F_t^2$ if $h$ is a \GBU. We do this explicitly for
	$F_t^2$,
	the proof for $F_t$ being slightly simpler.
	Direct calculation gives
	\begin{align*}
		(F_t^2)'(r)
			&= (h^2)'(\Gamma_t(r)) \Gamma_t'(r), \quad
		(F_t^2)''(r)
			= (h^2)''(\Gamma_t(r)) (\Gamma_t'(r))^2
				+ (h^2)'(\Gamma_t(r)) \Gamma_t''(r).
	\end{align*}
	But $(h^2)' \ge 0$ and $(h^2)'' \le 0$ because $h^2$ is
	increasing and concave.
	Also, $h$ concave implies that $\Gamma_t$ is concave: this is classical
	(see Lemma 8.3 of \cite{KFlow} for a proof).
	Hence, $\Gamma_t'' \le 0$, and we conclude that $(F_t^2)'' \le 0$,
	meaning that $F_t^2$ is concave.

	We now show that $h \le F_t \le C(t) h$.
	We have $F_t(r) = h(\Gamma_t(r)) \ge h(r)$
	since $\Gamma_t(r) \ge r$ and $h$ is increasing.
	Because $F_t$ is concave it is sublinear,
	so $F_t(r) \le c'r + d'$ for some $c', d'$ increasing in time.
	Hence,
	\begin{align*}
		F_t(r)
			&= h(\Gamma_t(r))
			\le h(c' r + d')
			\le 2 c' h(r) + h(d')
			\le (2 c' + h(d')/h(0)) h(r).
	\end{align*}
	Here, we used \cref{L:preGB} (we increase $c'$ so that $c' \ge 1$
	if necessary) and that $h$ increasing
	gives $h(d') \le (h(d')/h(0)) h(r)$. Hence,
	$h \le F_t \le C(t) h$, with $C(t)$ increasing with time.
	
	Finally, if $h$ is a \GBE then $C(t) \mu$ serves as
	a bound on the function $E$ of \cref{D:GrowthBound} for $F_t$.
\end{proof}

\Ignore{ 
\begin{lemma}\label{L:SubaddLowerBound}
	Let $h$ be a \GB. We have,
	\begin{align*}
		R[h]
			:= \sup_{x, y \in \R^2} \frac{h(x) + h(y)}{h(x + y)}
			= 2.
	\end{align*}
\end{lemma}
\begin{proof}
	We can re-express $R[h]$ as
	\begin{align*}
		R[h]
			= \sup_{r > 0} \sup_{s \in [0, r]} f_r(s), \qquad
		f_r(s)
			:= \frac{h(r - s) + h(s)}{h(r)}.
	\end{align*}
	Now, there is only one internal extremum of $f_r$, which occurs when
	\begin{align*}
		0
			&= f_r's(s)
			= \frac{-h'(r - s) + h'(s)}{h(r)}
		\iff
			h'(r - s) = h'(s)
		\iff s = \frac{r}{2},
	\end{align*}
	the last implication holding since $h'$ is monotone, since $h$ is
	concave. \ToDo{If $h'$ is not strictly concave then it could be
	that $h'$ is constant in a neighborhood of $r = s/2$.
	One solution is to require strict concavity in our \GBs,
	which would do no harm. Return to this.} 
	Hence,
	\begin{align*}
		R[h]
			= \sup_{r > 0} f_r\pr{\frac{r}{2}}
			= 2 \sup_{r > 0} \frac{h(r)}{h(2r)}
			= 2.
	\end{align*}
\end{proof}
} 

\begin{lemma}\label{L:fhzetaacts}
	Assume that $h$ is a \GB and let $g := 1/h$. Then
	$g$ is a decreasing convex function; in particular,
	$\abs{g'}$ is decreasing. Moreover
	\begin{align*}
		\abs{g'} \le c_0 g, \quad
		c_0 := h'(0)/h(0).
	\end{align*}
\end{lemma}
\begin{proof}
	We have
	\begin{align*}
		g'(r)
			= - \frac{h'(r)}{h(r)^2}
			< 0
	\end{align*}
	and
	\begin{align*}
		g''(r)
			= -\pr{\frac{h'(r)}{h^2(r)}}'
			= \frac{ 2 h(r) (h'(r))^2 - h^2(r) h''(x)}
				{h^4(r)}
			\ge 0,
	\end{align*}
	since $h > 0$ and $h'' \le 0$.
	Thus, $g$ is a decreasing convex function. Then, because
	$g'$ is negative but increasing, $\abs{g'}$ is decreasing.
	
	Finally,
	\begin{align*}
		\frac{\abs{g'(r)}}{g(r)}
			= \frac{h'(r)/h^2(r)}{1/h(r)}
			= \frac{h'(r)}{h(r)}
			= (\log h(r))'
	\end{align*}
	is decreasing, since $\log$ is concave and $h$ is concave so
	$\log h$ is concave. Therefore,
	\begin{align*}
		\abs{g'(r)} \le (\log h)'(0) g(r).
	\end{align*}
\end{proof}

\Ignore{ 
\begin{lemma}
	The function $M$ in \cref{D:GrowthBound} is convex with $M(0) = 0$.
\end{lemma}
\begin{proof}
	That $M(0) = 0$ is immediate. To show convexity, we work with the function,
	$M(r^2)$,
	calculating its derivatives in two different ways.
	First, we have
	\begin{align*}
		(M(r^2))'
			&= 2 r M'(r^2), \\
		(M(r^2))''
			&= 2 M'(r^2) + 4 r^2 M''(r^2).
	\end{align*}
	But,
	\begin{align*}
		M(r^2) = (1 + r H(r))^2 r^2,
	\end{align*}
	so also,
	\begin{align*}
		(M(r^2))'
			&= (1 + r H(r))^2 \, 2r
				+ 2 (1 + r H(r)) (H(r) + r H'(r)) r^2 \\
			&= 2 r (1 + r H(r))
				\pr{1 + r H(r) + r (H(r) + r H'(r))
				} \\
			&= 2 (r + r^2 H(r))
				\pr{1 + 2 r H(r) + r^2 H'(r)
				} \\
		(M(r^2))''
			&= 2 (1 + 2 r H(r) + r^2 H'(r))
				\pr{1 + 2 r H(r) + r^2 H'(r)
				} \\
			&\qquad
			+ 2 (r + r^2 H(r))
				\pr{2 H(r) + 2 r H'(r) + 2 r H'(r) + r^2 H''(r)
				} \\
			&= 2 (1 + 2 r H(r) + r^2 H'(r))^2
				+ 2 (r + r^2 H(r))
				\pr{2 H(r) + 4 r H'(r) + r^2 H''(r)
				}.
	\end{align*}
	But,
	\begin{align*}
		H'(r)
			&= \frac{h(r)}{r^2}, \\
		H''(r)
			&= - 2 \frac{h(r)}{r^2} + \frac{h'(r)}{r^2}
	\end{align*}
	so
	\begin{align*}
		2 H(r) + 4 r H'(r) + r^2 H''(r)
			= 2 H(r) - 8 \frac{h(r)}{r}  - 2 h(r) + h'(r)
	\end{align*}
\end{proof}
} 

We will also need the properties of $\omu$ (defined in  \cref{e:omu}) given in \cref{L:omuSubAdditive}.

\begin{lemma}\label{L:omuSubAdditive}
	For all $r \ge 0$ and $a \in [0,1]$,
	$
		a \omu(r)
			\le \omu(a r)
	$.
\end{lemma}
\begin{proof}
	As in the proof of \cref{L:hGood},
	because $\omu$ is concave with $\omu(0) = 0$, we have
	$a \omu(r) = a \omu(r) + (1 - a) \omu(0)
	\le \omu(a r + (1 - a) 0) = \omu(a r)$
	for all $r \ge 0$ and $a \in [0,1]$.
\end{proof}

\begin{remark}\label{R:mu}
	Similarly, $\mu$ of \cref{D:GrowthBound} (iv) satisfies
	$
		\mu(a r)
			\le a \mu(r)
	$
	for all $r \ge 0$ and $a\in[0,1]$.
\end{remark}

%
%
\section{Biot-Savart law and locally log-Lipschitz velocity fields}\label{S:BSLaw}

\noindent 

\begin{prop}\label{P:KBounds}
	Let $a_\la$ be as in \cref{D:Radial}.
	There exists $C > 0$ such that, for all $x \in \R^2$
	and all $\lambda > 0$ we have,
	\begin{align}
		\norm{a_\lambda(x - y) K(x - y)}_{L^1_y(\R^2)}
			&\le C \lambda.
					\label{e:aKBound}
		\end{align}

	Let $U \subseteq \R^2$ have Lebesgue measure $\abs{U}$.
	Then for any $p$ in $[1, 2)$,
	\begin{align}\label{e:RearrangementBounds}
		\begin{split}
		\smallnorm{K(x - \cdot)}_{L^p(U)}^p
			\le (2 \pi (2 - p))^{p - 2}
				\abs{U}^{1 - \frac{p}{2}}.
		\end{split}
	\end{align}
\end{prop}
\begin{proof}
	See Propositions 3.1 and 3.2 of \cite{AKLN2015}.
\end{proof}

\begin{prop}\label{P:gradgradaKBound}
	There exists $C > 0$ such that for all $\la > 0$,
	\begin{align*}
		\abs{\grad \grad^\perp((1 - a_\la) K)(x)}
			\le \frac{C}{\abs{x}^3}
				\CharFunc_{B_{\la/2}(0)^C}.
	\end{align*}
\end{prop}
\begin{proof}
	Using $\grad \grad (fg) = f \grad \grad g
	+ g \grad \grad f + 2 \grad f \otimes \grad g$, we have
	\begin{align*}
		\grad \grad^\perp((1 - a_\la) K)
			= (1 - a_\la) \grad \grad K
				+ K \grad \grad a_\la
				- 2 \grad a_\la \otimes \grad K.
	\end{align*}
	But, $\abs{K(x)} \le (2 \pi)^{-1} \abs{x}^{-1}$,
	$\abs{\grad K(x)} \le (2 \pi)^{-1} \abs{x}^{-2}$,
	$\abs{\grad \grad K(x)} \le (4 \pi)^{-1}
	\abs{x}^{-3}$, and
	\begin{align*}
		\abs{\grad a_\la(x)}
			&= \frac{1}{\la} \abs{\grad a(\la^{-1} x)}
			\le \frac{C}{\la}
				\CharFunc_{B_{\la}(0) \setminus
					B_{\la/2}(0)}(\la^{-1} x)
			\le \frac{C}{\abs{x}}
				\CharFunc_{B_{\la/2}(0)^C}, \\
		\abs{\grad \grad a_\la(x)}
			&= \frac{1}{\la^2}
				\abs{\grad \grad a(\la^{-1} x)}
			\le \frac{C}{\la^2}
				\CharFunc_{B_{\la}(0) \setminus
					B_{\la/2}(0)}(\la^{-1} x)
			\le \frac{C}{\abs{x}^2}
			\CharFunc_{B_{\la/2}(0)^C},
	\end{align*}
	from which the result follows.
\end{proof}

\cref{P:hlogBound} is a refinement of Proposition 3.3 of \cite{AKLN2015} that better accounts for the effect of the measure of $U$.

\begin{prop}\label{P:hlogBound}
	Let $X_1$ and $X_2$ be measure-preserving homeomorphisms of $\R^2$.
	Let $U \subset \R^2$ have finite measure and
	assume
	that $\delta := \norm{X_1 - X_2}_{L^\iny(U)} < \iny$. Then
	for any $x \in \R^2$,
	\begin{align}\label{e:KX1X2Diff}
		\begin{split}
				\smallnorm{K(x - X_1(z)) - K(x - X_2(z))}_{L^1_z(U)}
					&\le C R \omu(\delta/R)
		\end{split}
	\end{align}
	where $R = (2 \pi)^{-1/2}\abs{U}^{1/2}$.
\end{prop}
\begin{proof}
	As in the proof of Proposition 3.3 of \cite{AKLN2015}, we have
    \begin{align*}
        \norm{K(x-X_1(z)) - K(x-X_2(z))}_{L^1_z(U)}
        \le Cp R^{\frac{1}{p}} \delta^{1 - \frac{1}{p}}.
    \end{align*}
    In \cite{AKLN2015}, $R^{\frac{1}{p}}$ was bounded above by $\max \set{1, R}$,
    which gave $p = - \log \delta$ as the minimizer of the norm
    as long as $\delta < e^{-1}$. Keeping
    the factor of $R^{\frac{1}{p}}$ we see that the minimum occurs when
    $p = - \log(\delta/R)$ as long as $\delta \le e^{-1} R$, the minimum value being
	\begin{align*}
		- C \delta &\log(\delta/R)
				R^{-\frac{1}{\log(\delta/R)}}
				\delta^{\frac{1}{\log(\delta/R)}}
			= - C \delta \log(\delta/R)
					e^{-\frac{\log R}{\log(\delta/R)}}
					e^{\frac{\log \delta}{\log(\delta/R)}}
			= - C e^{-1} \delta \log(\delta/R) \\
			&= C R \omu(\delta/R).
	\end{align*}
	This gives the bound for $\delta \le e^{-1} R$;
	the $\delta > e^{-1} R$ bound follows immediately from
	\cref{e:RearrangementBounds} with $p = 1$.
\end{proof}

In \cref{P:Morrey}, we establish a bound on the \MOC of $u \in S_h$.
In Lemma 2.8 of \cite{ElgindiJeong2016}, the authors obtain the same bound as in \cref{P:Morrey} for $h(x) = 1 + \abs{x}$, but under the assumption that the velocity field can be obtained from the vorticity via a symmetrized Biot-Savart law (which they show applies to $m$-fold symmetric vorticities for $m \ge 3$, but which does not apply for our unbounded velocities).

\begin{prop}\label{P:Morrey}
	Let $h$ be a \preGB.
	Then for all $x, y \in \R^2$ such that
	$\abs{y} \le C(1 + \abs{x})$ for some constant $C > 0$, we have,
	for all $u \in S_h$,
	\begin{align*}
		\abs{u(x + y) - u(x)}
			\le
				C \norm{u}_{S_h } h(x) \omu \pr{\frac{\abs{y}}{h(x)}}.
	\end{align*}
	If $h \equiv C$, we need no restriction on $\abs{y}$.
\end{prop}
\begin{proof}
Fix $x \in \R^2$ and let $\psi$ be the stream function for $u$ on $\R^2$ chosen so that $\psi(x) = 0$.
Let $\phi = a_2$, so that $\supp \phi \subseteq \ol{B_{2}(0)}$ with $\phi \equiv 1$ on $\ol{B_{1}(0)}$ and let $\phi_{x, R}(y) := \phi(R^{-1} (y - x))$ for any $R > 0$. Let $\ol{u} = \grad^\perp(\phi_{x, R} \psi)$ and let $\ol{\omega} = \curl \ol{u}$.

Applying Morrey's inequality gives, for any $y$ with $\abs{y} \le R$, and any $p > 2$,
\begin{align}\label{e:uuMorrey}
	\abs{u(x + y) - u(x)}
		= \abs{\ol{u}(x + y) - \ol{u}(x)}
		\le C \norm{\grad \ol{u}}_{L^p(\R^2)}
			\abs{y}^{1 - \frac{2}{p}}.
\end{align}
Because $\ol{\omega}$ is compactly supported, $\ol{u} = K * \ol{\omega}$. Thus, we can apply the Calderon-Zygmund inequality to obtain
\begin{align*}
	&\abs{u(x + y) - u(x)}
		\le C \inf_{p > 2} \set{p \norm{\ol{\omega}}_{L^p(\R^2)} \abs{y}^{1 - \frac{2}{p}}} \\
		&\qquad
		= C \inf_{p > 2} \set{p \norm{\ol{\omega}}_{L^p(B_{2 R}(x))} \abs{y}^{1 - \frac{2}{p}}}
		\le C \norm{\ol{\omega}}_{L^\iny(\R^2)}
			\inf_{p > 2} \set{R^{2/p} p \abs{y}^{1 - \frac{2}{p}}} \\
		&\qquad
		= C \abs{y} \norm{\ol{\omega}}_{L^\iny(\R^2)}
			\inf_{p > 2} \set{p \abs{R^{-1} y}^{- \frac{2}{p}}}.
\end{align*}
When $R^{-1} \abs{y} \le e^{-1}$ (meaning also that $\abs{y} \le R$, as required), the infimum occurs at $p = - 2 \log(\abs{R^{-1} y})$ and we have
\begin{align*}
	&\abs{u(x + y) - u(x)}
		\le - C \norm{\ol{\omega}}_{L^\iny(\R^2)} \abs{y} \log \abs{R^{-1} y}.
\end{align*}

Having minimized over $p$ for a fixed $R$, we must now choose $R$.
First observe that $\ol{\omega} = \Delta (\phi_{x, R} \psi) = \phi_{x, R} \Delta \psi + \Delta \phi_{x, R} \psi + 2 \grad \phi_{x, R} \cdot \grad \psi = \phi_{x, R} \omega + \Delta \phi_{x, R} \psi + 2 \grad^{\perp} \phi_{x, R} \cdot u$. Also, for all $z \in B_{2R}(x)$
\begin{align*}
	\abs{\psi(z)}
		\le \int_{\abs{x}}^{\abs{x} + 2R} \norm{g u}_{L^\iny} h(r) \, dr
		\le \norm{g u}_{L^\iny} R h(\abs{x} + 2R),
\end{align*}
where $g := 1/h$. Hence,
\begin{align}\label{e:olomega}
	\begin{split}
	\norm{\ol{\omega}}_{L^\iny(\R^2)}
		&\le \norm{\phi_{x, R}}_{L^\iny} \smallnorm{\omega}_{L^\iny(B_{2 R}(x))}
			+ 2 \smallnorm{\grad^\perp \phi_{x, R}}_{L^\iny}
				\smallnorm{u}_{L^\iny(B_{2 R}(x))} \\
		&\qquad
			+ \norm{g u}_{L^\iny} h(\abs{x} + 2R)
				R \smallnorm{\Delta \phi_{x, R}}_{L^\iny(B_{2 R}(x))} \\
		&\le \smallnorm{\omega}_{L^\iny(\R^2)}
			+ C R^{-1} h(\abs{x} + 2 R)
				\smallnorm{gu}_{L^\iny(\R^2)}
	\end{split}
\end{align}
so
\begin{align*} 
	&\abs{u(x + y) - u(x)}
		\le - C \brac{\smallnorm{\omega}_{L^\iny(\R^2)}
					+ R^{-1} h(\abs{x} + 2 R) \smallnorm{gu}_{L^\iny(\R^2)}}
				\abs{y} \log \abs{R^{-1} y}.
\end{align*}

Now choose $R = h(x)$. Then
\begin{align*}
	R^{-1} h&(\abs{x} + 2 R)
		= g(x) h(\abs{x} + 2h(x))
		\le C g(x) h(\abs{x}) + h(2h(x)) \\
		&\le C(1 + C g(x) h(h(x)))
		\le C,
\end{align*}
where we used the subadditivity of $h$ (see \cref{L:hGood}) and \cref{L:preGB}.
Hence, we have, for $\abs{y} \le e^{-1} h(x)$,
\begin{align*}
	&\abs{u(x + y) - u(x)}
		\le - C \norm{u}_{S_h}
				\abs{y} (\log \abs{y} + \log g(x)) \\
		&\qquad
		= C \norm{u}_{S_h} \abs{y} \log \pr{\frac{h(x)}{\abs{y}}}
		= C \norm{u}_{S_h} h(x) \omu \pr{\frac{\abs{y}}{h(x)}},
\end{align*}
the last equality holding as long as $\abs{y} \le h(x) e^{-1} = e^{-1} R \le R$.

If $\abs{y} \ge e^{-1} R$, then $\omu \pr{\frac{\abs{y}}{h(x)}} = e^{-1}$, and
\begin{align*}
	\abs{u(x + y) - u(x)}
		&\le C \norm{u}_{S_h} (h(x) + h(x + y))
		\le C \norm{u}_{S_h} h(x) \\
		&= C \norm{u}_{S_h} h(x) \omu \pr{\frac{\abs{y}}{h(x)}},
\end{align*}
where we applied \cref{L:preGB}, using $\abs{y} \le C(1 + \abs{x})$. Note that if $h \equiv C$, however, we need no restriction on $\abs{y}$ to reach this conclusion, since $h(x) = h(x + y) = C$.
\end{proof}

\Ignore{ 
\begin{remark}\label{R:LL}
	In \cref{S:ContDep,S:LP} we will make much use of $S_1$.
	As we can see from \cref{P:Morrey}, any $u \in S_1$
	also lies in the space $LL(\R^2)$ of log-Lipschitz
	functions---those functions having norm
	\begin{align*}
		\norm{u}_{LL}
			:= \norm{u}_{L^\iny}
				+ \sup_{\substack{x, y \in \R^2 \\ x \ne y}}
					\frac{\abs{u(x) - u(y)}}{\omu(\abs{x - y})}.
	\end{align*}
\end{remark}
} 

In proving uniqueness in \cref{S:Uniqueness}, we will need to bound the term in the Serfati identity \cref{e:SerfatiId} coming from a convolution of the difference between two vorticities. Since the vorticities have no assumed regularity, we will need to rearrange the convolution so as to use an estimate on the Biot-Savart kernel that involves the difference of the flow maps, as in \cref{P:olgKBound}. This proposition is a refinement of Proposition 6.2 of \cite{AKLN2015} that better accounts for the effect of the parameter $\la$ in the cutoff function $a_\la$. Note that although we assume the solutions lie in some $S_h$ space, $h$ does not appear directly in the estimates, rather it appears indirectly via the value of $\delta(t)$, as one can see in the application of the proposition.

\begin{prop}\label{P:olgKBound}
	Let $X_1$ and $X_2$ be measure-preserving homeomorphisms of $\R^2$
	and let $\omega^0 \in L^\iny(\R^2)$.
	Fix $x \in \R^2$ and $\la > 0$.
	Let $V = \supp a_\la(X_1(s, x) - X_1(s, \cdot))
	 \cup \supp a_\la(X_1(s, x) - X_2(s, \cdot))$ and assume that
	\begin{align}\label{e:deltaVDef}
		\delta(t)
			&:= \norm{X_1(t, \cdot) - X_2(t, \cdot)}_{L^\iny(V)}
				< \iny.
	\end{align}
	Then we have
	\begin{align*}
		\abs{\int (a_\la K(X_1(s, x) - X_1(s, y))
					- a_\la K(X_1(s, x) - X_2(s, y))) \omega^0(y) \, dy}
				\le C \smallnorm{\omega^0}_{L^\iny}
					\la \omu(\delta(t)/\la).
	\end{align*}
	The constant, $C$, depends only on the Lipschitz constant of $a$.
\end{prop}
\begin{proof}
	We have,
	\begin{align*}
		\int &(a_\la K(X_1(s, x) - X_1(s, y))
				- a_\la K(X_1(s, x) - X_2(s, y))) \omega^0(y) \, dy
			= I_1 + I_2,
	\end{align*}
	where
	\begin{align*}
		I_1
			&:= \int a_\la(X_1(s, x) - X_1(s, y))
					\pr{K(X_1(s, x) - X_1(s, y)) - K(X_1(s, x) - X_2(s, y))}
					\omega^0(y) \, dy, \\
		I_2
			&:= \int \pr{a_\la (X_1(s, x) - X_1(s, y))
				- a_\la(X_1(s, x) - X_2(s, y))}
					K(X_1(s, x) - X_2(s, y)) \omega^0(y) \, dy.
	\end{align*}
	
	To bound $I_1$, let $U = \supp a_\la(X_1(s, x) - X_1(s, \cdot)) \subseteq V$,
	which we note
	has measure $4 \pi \la^2$ independently of $x$. Then
	\begin{align*}
		\abs{I_1}
			&\le \norm{K(X_1(s, x) - X_1(s, y)) - K(X_1(s, x) - X_2(s, y))}
					_{L^1_y(U)}
				\smallnorm{\omega^0}_{L^\iny(\R^2)} \\
			&\le C \la \smallnorm{\omega^0}_{L^\iny(\R^2)} \omu(\delta(t)/\la).
	\end{align*}
	Here, we applied \cref{P:hlogBound} at the point $X_1(s, x)$.
	
	For $I_2$,
	we have,
	\begin{align*}
		\abs{I_2}
			&\le \int \abs{\pr{a_\la (X_1(s, x) - X_1(s, y))
						- a_\la(X_1(s, x) - X_2(s, y))}
					K(X_1(s, x) - X_2(s, y)) \omega^0(y)} \, dy \\
			&\le \frac{C}{\la}
				\int_V \abs{X_1(s, y)) - X_2(s, y))}
					\abs{K(X_1(s, x) - X_2(s, y))} \smallabs{\omega^0(y)} \, dy \\
			&\le \frac{C}{\la} \smallnorm{\omega^0}_{L^\iny} \delta(t) 
				\int_V 
					\abs{K(X_1(s, x) - X_2(s, y))}\, dy
			\le C \smallnorm{\omega^0}_{L^\iny} \delta(t).
	\end{align*}
	Here, we used \cref{e:RearrangementBounds} with $p = 1$
	and that the Lipschitz constant
	of $a_\la$ is $C \la^{-1}$. Also, though,
	\begin{align*}
		\abs{I_2}
			&\le 2 \smallnorm{\omega^0}_{L^\iny} \int_V 
					\abs{K(X_1(s, x) - X_2(s, y))} \, dy
			\le C \la \smallnorm{\omega^0}_{L^\iny},
	\end{align*}
	again using \cref{e:RearrangementBounds}.
	The result then follows from observing that
	 $\delta(t) \le \la \omu(\delta(t)/\la)$ for $\delta(t) \le \la e^{-1}$
	and $\omu(\delta(t)/ \la) = e^{-1}$ for $\delta(t) \ge \la e^{-1}$.
\end{proof}

%
%
\section{Flow map bounds}\label{S:FlowMap}

\noindent In this section we develop bounds related to the flow map for solutions to the Euler equations in $S_h$ on $[0, T]$. First, though, is the matter of existence and uniqueness:

\begin{lemma}\label{L:FlowWellPosed}
	Let $h$ be a \preGB (which we note includes $h(x) = C (1 + \abs{x})$)
	and assume that $u \in L^\iny(0, T; S_h)$. Then there exists a unique
	flow map, $X$, for $u$; that is, a function
	$X \colon [0, T] \times \R^2 \to \R^2$ for which
	\begin{align*}
		X(t, x)
			= x + \int_0^t u(s, X(s, x)) \, ds
	\end{align*}
	for all $(t, x) \in [0, T] \times \R^2$.
\end{lemma}
\begin{proof}
	Because $u$ is locally log-Lipschitz by \cref{P:Morrey},
	this is (essentially) classical.
\end{proof}

\begin{lemma}\label{L:FlowBounds}
	Let $h$ be a \preGB.
	Assume that $u_1, u_2 \in L^\iny(0, T; S_h)$.
	Let $F_t$ be the function defined in \cref{e:FDef}.
	We have,
	\begin{align}\label{e:X1X2xBoundTime}
		\begin{array}{ll}
			\displaystyle \frac{\abs{X_1(t, x) - X_2(t, x)}}{F_t(x)}
				\le C_0 t, 
			&\displaystyle \frac{\abs{X_j(t, x) - x}}{F_t(x)}
				\le C_0 t,
		\end{array}
	\end{align}
	where $C_0 = \norm{u_j}_{L^\iny(0, T; S_h)}$.
\end{lemma}
\begin{proof}
	For $j = 1, 2$,
	$\abs{u_j(t, x)} \le \norm{u_j(t)}_{S_h} \abs{h(x)}$,
	so
	\begin{align*}
		\abs{X_j(t, x)}
			&\le \abs{x} + \int_0^t \abs{u_j(s, X_j(s, x))} \,ds
			\le \abs{x} + C_0 \int_0^t h(X_j(s,x)) \,ds.
	\end{align*}
	Hence by Osgood's inequality,
	$\abs{X_j(t, x)} \le \Gamma_t(x)$,
	where $\Gamma_t$ is defined in \cref{e:GammaDef}.
	We also have
	\begin{align*}
		\abs{X_j(t, x) - x}
			&\le \int_0^t \abs{u_j(s, X_j(s, x))} \,ds
			\le C_0 \int_0^t h(X_j(s, x)) \,ds
			\le C_0 t F_t(x).
	\end{align*}
	
	Similarly,
	\begin{align*}
		\abs{X_1(t, x) - X_2(t, x)}
			&\le \int_0^t \pr{\abs{u_1(s, X_1(s, x))}
				+ \abs{u_2(s, X_2(s, x))}} \, ds
			\le C_0 t F_t(x).
	\end{align*}
	These bounds yield the result.
\end{proof}

\cref{L:FProperties,L:FlowBounds} together show that over time, the flow transports a ``particle'' of fluid at a distance $r$ from the origin by no more than a constant times $h(r)$. This will allow us to control the growth at infinity of the velocity field over time so that it remains in $S_h$ (for at least a finite time), as we shall see in the next section. As the fluid evolves over time, however, the flow can move two points farther and farther apart; that is, its spatial \MOC can worsen, though in a controlled way, as we show in \cref{L:FlowUpperSpatialMOC}.
(A similar bound to that in \cref{L:FlowUpperSpatialMOC} holds for any growth bound, but we restrict ourselves to the special case of bounded vorticity, bounded velocity velocity fields, for that is all we will need.)

\begin{lemma}\label{L:FlowUpperSpatialMOC}
	Let $u \in L^\iny(0, T; S_1)$
	and let $X$ be the unique flow map for $u$.
	Let $C_0 = \norm{u}_{L^\iny(0, T; S_1)}$.
	For any $t \in [0, T]$  define the function,
	\begin{align*}
		\chi_t(r)
			&:=
			\left\{
			\begin{array}{ll}
				r^{e^{-C_0 t}}
					& \text{when } r \le 1 \\
				r
					& \text{when } r > 1
			\end{array}
			\right\}
			\le r + r^{e^{-C_0 t}}.
	\end{align*}
	Then for all $x, y \in \R^2$,
	\begin{align*}
		\abs{X(t, x) - X(t, y)}
			&\le C(T) \chi_t(\abs{x - y}).
	\end{align*}
	The same bound holds for $X^{-1}$.
\end{lemma}
\begin{proof}
	The bounds,
	\begin{align*}
		\abs{X(t, x) - X(t, y)}, \abs{X^{-1}(t, x) - X^{-1}(t, y)}
			\le
				C(T) \abs{x - y}^{e^{-C_0 t}}
	\end{align*}
	are established in  Lemma 8.2 of \cite{MB2002}. We note, however,
	that that proof applies only for all sufficiently small $\abs{x - y}$.
	A slight refinement of the proof produces the bounds as we have
	stated them.
\end{proof}

The following simple bound will be useful later in the proof of \cref{L:ForJ1Bound}:
\begin{align}\label{e:chitBound}
	\chi_t(a r)
	\le a^{e^{-C_0 t}} \chi_t(r)
		\text{ for all } a \in [0, 1], r > 0.
\end{align}

\section{Existence}\label{S:Existence}

\noindent Our proof of existence differs significantly from that in \cite{AKLN2015} only in the use of the Serfati identity to obtain a bound in $L^\iny(0, T; S_h)$ of a sequence of approximating solutions and to show that the sequence is Cauchy, which is more involved than in \cite{AKLN2015}. Although velocities in $S_h$ are not log-Lipschitz in the whole plane (unless $h$ is constant), they are log-Lipschitz in any compact subset of $\R^2$. Since the majority of the proof of existence involves obtaining convergence on compact subsets, this has little effect on the proof. Therefore, we give only the details of the bound on $L^\iny(0, T; S_h)$ using the Serfati identity, as this is the main modification of the existence proof. We refer the reader to \cite{AKLN2015} for the remainder of the argument.


\begin{proof}[\textbf{Proof of existence in \cref{T:Existence}}]
Let $u^0 \in S_h$ and assume that $u^0$ does not vanish identically; otherwise, there is nothing to prove.
Let $(u_n^0)_{n=1}^\iny$ and $(\omega_n^0)_{n=1}^\iny$ be compactly supported approximating sequences to the initial velocity, $u^0$, and initial vorticity, $\omega^0$, obtained by cutting off the stream function and mollifying by a smooth, compactly supported mollifier. (This is as done in Proposition B.2 of \cite{AKLN2015}, which simplifies tremendously when specializing to all of $\R^2$.)
Let $u_n$ be the classical, smooth solution to the Euler equations with initial velocity $u_n^0$, and note that its vorticity is compactly supported for all time. The existence and uniqueness of such solutions follows, for instance, from \cite{McGrath1967} and references therein. (See also Chapter 4 of \cite{MB2002} or Chapter 4 of \cite{C1998}.) Finally, let $\omega_n = \curl u_n$.

As we stated above, we give only the uniform $L^\iny([0, T]; S_h)$ bound for this sequence, the rest of the proof differing little from that in \cite{AKLN2015}.

We have,
\begin{align}\label{e:omegaunLInfBound}
	\norm{u_n^0}_{S_h} \le C \norm{u^0}_{S_h}.
\end{align}

It follows as in Proposition 4.1 of \cite{AKLN2015} that the Serfati identity \cref{e:SerfatiId} holds for the approximate solutions. It is important to note that $x$ and $t$ are fixed in this identity, so $\la$ can be a function both of $t$ and $x$ (though not $s$). Or, to see this more expicitly, we can write the critical convolution in the derivation of the Serfati identity in Proposition 4.1 of \cite{AKLN2015} as,
\begin{align*}
    ((1 - a_{\la(t, x)}) &K^j) * \prt_s \omega(x)
        = \int_{\R^2} ((1 - a_{\la(t, x)}(x - y))
        		K^j(x - y))  \prt_s \omega(y) \, dy,
\end{align*}
and it becomes clear that in moving derivatives from one side of the convolution to another we are in effect integrating by parts, taking derivatives always in the variable $y$.

In any case, it follows from the Serfati identity that
\begin{align*}
	\abs{u_n(t, x)}
		\le &\abs{u_n^0(x)}
			+ \abs{(a_\lambda K) *(\omega_n(t) - \omega_n^0)(x)} \\
		&
			+ \int_0^t \abs{\pr{\grad \grad^\perp \brac{(1 - a_\lambda) K}}
        \stardot (u_n \otimes u_n)(s, x)} \, ds.
\end{align*}

The first convolution we bound using \cref{e:aKBound} and \cref{e:omegaunLInfBound}
as
\begin{align*}
	\abs{(a_\la K) *(\omega_n(t) - \omega_n^0)(x)}
		\le C \pr{
			\lambda\smallnorm{\omega_n(t)}_{L^\iny(B_\la(x))}
			+ \lambda\smallnorm{\omega_n^0}_{L^\iny(B_\la(x))}}
		\le C \la.
\end{align*}

For the second convolution, we have, using \cref{P:gradgradaKBound},
\begin{align*} 
	\begin{split}
	&\abs{\pr{\grad \grad^\perp
			\brac{(1 - a_\la) K}}
        		\stardot (u_n \otimes u_n)(s)}
		\le \int_{B_{\la/2}(x)^C} \frac{C}{\abs{x - y}^3}
        		\abs{u_n(s, y)}^2 \, dy \\
		&\qquad
		= C \int_{B_{\la/2}(x)^C} \frac{h(y)^2}{\abs{x - y}^3}
        		\abs{\frac{u_n(s, y)}{h(y)}}^2 \, dy \\
		&\qquad
		\le C \norm{\frac{u_n(s)}{h}}_{L^\iny}^2
			\brac{
			\int_{B_{\la/2}(x)^C} \frac{h (x - y)^2}
				{\abs{x - y}^3} \, dy
			+ h(x)^2 \int_{B_{\la/2}(x)^C}
				\frac{1}{\abs{x - y}^3} \, dy
			} \\
		&\qquad
		= C \norm{\frac{u_n(s)}{h}}_{L^\iny}^2
			\brac{H(\la(x)/2) + C \frac{h(x)^2}{\la(x)}
			},
	\end{split}
\end{align*}
where $H = H[h^2]$ is defined in \cref{e:HDef}.
The second inequality
follows from the subadditivity of $h^2$ (as in \cref{R:hSubadditive}).
Hence,
\begin{align*}
	\abs{u_n(t, x)}
		\le &\abs{u_n^0(x)}
			+ C \lambda(x)
			+ C \int_0^t \norm{\frac{u_n(s)}{h}}_{L^\iny}^2
			\brac{
			H(\la(x)/2) + C \frac{h(x)^2}{\la(x)}
			} \, ds.
\end{align*}
Dividing both sides by $h(x)$ gives
\begin{align}\label{e:KeyExistenceBound}
	\begin{split}
	\abs{\frac{u_n(t, x)}{h(x)}}
		\le &\abs{\frac{u_n^0(x)}{h(x)}}
			+ C \frac{\lambda(x)}{h(x)}
			+ C \int_0^t \norm{\frac{u_n(s)}{h}}_{L^\iny}^2
			\brac{
			\frac{H(\la(x)/2)}{h(x)} + C \frac{h(x)}{\la(x)}
			} \, ds.
	\end{split}
\end{align}

Now, for any fixed $t$, we can set
\begin{align}\label{e:laChoice}
	\la
		= \la(t, x)
		= 2 h(x)
		\pr{\int_0^t \norm{\frac{u_n(s)}{h}}_{L^\iny}^2
			\, ds}
			^{\frac{1}{2}},
\end{align}
which we note nearly minimizes the right-hand side of \cref{e:KeyExistenceBound}.
Defining
\begin{align*}
	\Lambda(s)
		:= \norm{\frac{u_n(s)}{h}}_{L^\iny}^2,
\end{align*}
this leads to
\begingroup
\allowdisplaybreaks
\begin{align*}
	\abs{\frac{u_n(t, x)}{h(x)}}
		\le &\abs{\frac{u_n^0(x)}{h(x)}}
			+ C \smallnorm{\omega^0}_{L^\iny}
				\pr{\int_0^t \Lambda(s) \, ds}^{\frac{1}{2}} \\
		&\qquad
		+ C \int_0^t \Lambda(s) H \pr{h(x) \pr{\int_0^t \Lambda(r) \, dr}^{\frac{1}{2}}
				}g(x) \, ds
			+ C \pr{\int_0^t \Lambda(s) \, ds}^{\frac{1}{2}} \\
		&\le C 
			+ C g(x)H \pr{h(x) \pr{\int_0^t \Lambda(s) \, ds}^{\frac{1}{2}}}
					\int_0^t \Lambda(s) \, ds
			+ C \pr{\int_0^t \Lambda(s) \, ds}^{\frac{1}{2}} \\
		&\le C 
			+ C H \pr{\pr{\int_0^t \Lambda(s) \, ds}^{\frac{1}{2}}}
					\int_0^t \Lambda(s) \, ds
			+ C \pr{\int_0^t \Lambda(s) \, ds}^{\frac{1}{2}} \\
		&\le C + C \pr{\int_0^t \Lambda(s) \, ds}^{\frac{1}{2}}
			+ C f \pr{\pr{\int_0^t \Lambda(s) \, ds}^{\frac{1}{2}}}
					\pr{\int_0^t \Lambda(s) \, ds}^{\frac{1}{2}} \\
		&\le C
			+ C \brac{1 + f \pr{\pr{\int_0^t \Lambda(s) \, ds}^{\frac{1}{2}}}}
					\pr{\int_0^t \Lambda(s) \, ds}^{\frac{1}{2}},
\end{align*}
\endgroup
where $f(z) := z H(z)$.
In the third-to-last inequality we used that $H$ is decreasing and $h(x) \ge h(0) > 0$.

Observe that although this inequality was obtained by choosing $\la = \la(t, x)$ for one fixed $t$, the inequality itself holds for all $t \in [0, T]$.

Taking the supremum over $x \in \R^2$ and squaring both sides, we have
\begin{align}\label{e:LBoundM}
	\Lambda(t)
		\le C + E \pr{\int_0^t \Lambda(s) \, ds}
		\le C + \mu \pr{\int_0^t \Lambda(s) \, ds},
\end{align}
where $E$, $\mu$ are as in \cref{D:GrowthBound}.
Now we can apply \cref{L:EProp,L:NotOsgood} to conclude that $u_n \in L^\iny(0, T; S_h)$ with a norm bounded uniformly over $n$. \cref{L:NotOsgood} also gives global-in-time existence ($T$ arbitrarily large) when \cref{e:omuOsgoodAtInfinity} holds.
\end{proof}

\begin{lemma}\label{L:NotOsgood}
	Assume that $\Lambda \colon [0, \iny) \to [0, \iny)$ is continuous with
	\begin{align}\label{e:LBound}
		\Lambda(t)
			\le \Lambda_0 + \mu \pr{\int_0^t \Lambda(s) \, ds}
	\end{align}
	for some $\Lambda_0 \ge 0$,
	where $\mu \colon [0, \iny) \to [0, \iny)$ is convex.
	Then for all $t \le 1$,
	\begin{align}\label{e:t1Bound}
		\int_{\Lambda_0}^{\Lambda(t)} \frac{ds}{\mu(s)}
			\le t
	\end{align}
	and for all $t \in [0, T]$ for any fixed $T \ge 1$,
	\begin{align}\label{e:TLargeBound}
		\int_{\Lambda_0}^{\Lambda(t)} \frac{ds}{\mu(T s)}
			\le \frac{t}{T}.
	\end{align}
	Moreover, if
	\begin{align}\label{e:NotOsgood}
		\int_1^\iny \frac{ds}{\mu(s)} = \iny
	\end{align}
	then $\Lambda \in L^\iny_{loc}([0, \iny))$.
\end{lemma}
\begin{proof}
	Because $\mu$ is convex,
	we can apply Jensen's inequality 
	to conclude that
	\begin{align*}
		\Lambda(t)
			\le \Lambda_0 + \mu \pr{\int_0^t t \Lambda(s) \, \frac{ds}{t}}
			\le \Lambda_0 + \int_0^t \mu(t \Lambda(s)) \, \frac{ds}{t}.
	\end{align*}
	As long as $t \le 1$, \cref{R:mu} allows us to write
	\begin{align*}
		\Lambda(t)
			\le \Lambda_0 + \int_0^t \mu(\Lambda(s)) \, ds,
	\end{align*}
	and Osgood's lemma gives \cref{e:t1Bound}.
	Now suppose that $T > 1$. Then \cref{R:mu} gives the weaker bound,
	\begin{align*}
		\mu(t \Lambda(s))
			= \mu \pr{\frac{t}{T} T \Lambda(s)}
			\le \frac{t}{T} \mu \pr{T \Lambda(s)}
	\end{align*}
	so that
	\begin{align*}
		\Lambda(t)
			\le \Lambda_0 + \frac{1}{T} \int_0^t \mu(T \Lambda(s)) \, ds,
	\end{align*}
	leading to \cref{e:TLargeBound}.
	
	Finally, if \cref{e:NotOsgood} holds then applying
	Osgood's lemma to \cref{e:TLargeBound} shows that
	$\Lambda$ is bounded on any interval $[0, T]$,
	so that $\Lambda \in L^\iny_{loc}([0, \iny))$.
\end{proof}

We make a few remarks on our proof of \cref{T:Existence}.

	\cref{L:NotOsgood} allows us to obtain finite-time or global-in-time
	existence of solutions, but unless we have a stronger condition on $\mu$,
	neither the finite time of existence
	nor the bound on the growth of the $L^\iny$ norm that results will
	be optimal. For both of our example growth bounds in \cref{C:MainResult}
	there are stronger conditions; namely, if $\mu_1$, $\mu_2$ are the function
	in \cref{D:GrowthBound} corresponding to $h_1$, $h_2$ then
	for all $a, r \ge 0$,
	\begin{align*} 
		\mu_1(a r) \le C_0 a^{1 + \al} \mu_1(r), \quad
		\mu_2(a r) \le C_0 a \mu_2(r).
	\end{align*}
	It is easy to see that the condition on $\mu_2$ in fact, implies
	\cref{e:NotOsgood}, though the condition on $\mu_1$ is too weak
	to do so. Both conditions improve the bound on the $L^\iny$ norm
	resulting from \cref{L:NotOsgood} and for $h_1$, the time of existence.

Moreover, \cref{L:EProp} shows that, up to a constant factor, $\mu(r) := C r(1 + r)$ works for all \GBUs (and gives $\mu_2(a r) \le C_0 a^2 \mu_2(r)$ for all $a, r \ge 0$). This \textit{suggests} that a slight weakening of the condition we placed on \GBs in $(ii)$ of \cref{D:GrowthBound} could be made that would still allow finite-time existence to be obtained.

\Ignore{ 
	On the surface, \cref{e:omuOsgoodAtInfinity} looks like Osgood's lemma
	(see Lemma 5.2.1 of \cite{C1998}), but it is not:
	The function $\mu$ in \cref{e:LBound} appears outside rather
	than inside the integral and the integration in \cref{e:NotOsgood} is from
	$1$ to $\iny$ rather than from $0$ to $1$.
	Also, in a typical application of Osgood's
	lemma $\mu$ is a concave modulus of continuity rather
	than a convex function.
	As we see from its proof, however, it follows
	as a corollary of Osgood's lemma.
} 


%
%
\section{Uniqueness}\label{S:Uniqueness}

\noindent
In this section we prove \cref{T:ProtoUniqueness}, from which uniqueness immediately follows. Our argument is a an adaptation of the approach of Serfati as it appears in \cite{AKLN2015}. It starts, however, by exploiting the flow map estimates in \cref{L:FlowBounds}, inspired by the proof of Lemma 2.13 of \cite{ElgindiJeong2016}, which is itself an adaptation of Marchioro's and Pulvirenti's elegant uniqueness proof for 2D Euler in \cite{MP1994}, in which a weight is introduced.

\begin{proof}[\textbf{Proof of \cref{T:ProtoUniqueness}}]
	We will use the bound on $X_1$ and $X_2$ given
	by \cref{L:FlowBounds} with the \GB,
	$F_T[\zeta]$, defined in \cref{e:FDef}.
	This is valid since $\zeta \ge  h$.
	By \cref{L:FProperties}, $F_T[\zeta]$ is a \GB
	that is equivalent to $\zeta$, up to a factor of $C(T)$;
	hence, we will use $\zeta$ in place of $F_T[\zeta]$,
	which will simply introduce a factor $C(T)$ into our
	bounds.
	
	By the expression for the flow maps in \cref{L:FlowWellPosed},
	\begin{align*}
		\eta(t)
			\le \norm{
				\int_0^t \frac{\abs{u_1(s, X_1(s, x)) - u_2(s, X_2(s, x))}}
					{\zeta(x)} \, ds
				}
			\le \int_0^t L(s) \, ds
			= M(t).
	\end{align*}
	We also have,
	\begin{align*}
		L(s)
			\le \norm{A_1(s, x)}_{L^\iny_x(\R^2)}
				+ \norm{A_2(s, x)}_{L^\iny_x(\R^2)},
	\end{align*}
	where
	\begin{align*}
		A_1(s, x)
			:= \frac{u_2(s, X_1(s, x)) - u_2(s, X_2(s, x))}{\zeta(x)}, \\
		A_2(s, x)
			:= \frac{u_1(s, X_1(s, x)) - u_2(s, X_1(s, x))}{\zeta(x)}.
	\end{align*}

	For $A_1$, first observe that \cref{L:FlowBounds} shows that
	$\abs{X_1(s, x) - X_2(s, x)} \le Ct \zeta(x) \le Ct (1 + \abs{x})$ and
	$\abs{X_1(s, x) - x} \le Ct \zeta(x) \le Ct (1 + \abs{x})$.
	Hence, we can apply \cref{P:Morrey} with $\zeta$ in place of
	$h$ to give
	\begin{align*}
		&\abs{u_2(s, X_1(s, x)) - u_2(s, X_2(s, x))}
			\le C \norm{u_2}_{S_\zeta} \zeta(x)
				\omu \pr{\abs{X_1(s, x) - X_2(s, x)}/\zeta(x)} \\
			&\qquad
			\le C \zeta(x) \omu(\eta(s)).
	\end{align*}
	It follows that
	\begin{align}\label{e:A1Bound}
		\abs{A_1(s, x)}
			&\le C \frac{\zeta(x)}{\zeta(x)} \omu(\eta(s))
			\le C \omu(\eta(s)).
	\end{align}	

	To bound $A_2$, we use the Serfati identity, choosing $\la(x) = h(x)$, to write
	\begin{align*}
		\abs{A_2(s, x)}
			\le \frac{\abs{u_1^0(x) - u_2^0(x)}}{\zeta(x)}
				+ A_2^1(s, x) + A_2^2(s, x),
	\end{align*}
	where
	\begin{align*}
		A_2^1(s, x)
			&= 
				\frac{1}{\zeta(x)}
					\abs{(a_{h(x)} K) * (\omega_1(s) - \omega_2(s))(X_1(s, x))
			- (a_{h(x)} K) * (\omega_1^0 - \omega_2^0)(X_1(s, x))}, \\
		A_2^2(s, x)
			&= \frac{1}{\zeta(x)}
				\abs{\int_0^s
				\pr{\grad \grad^\perp [(1 - a_{h(x)}) K]
					\stardot (u_1 \otimes u_1 - u_2 \otimes u_2)}(r, X_1(r, x))
					\, dr}.
	\end{align*}

	We write,
		\begin{align*}
		(a_{h(x)} &K) * (\omega_1(s) - \omega_2(s))(X_1(s, x)) \\
			&= \int (a_{h(x)} K(X_1(s, x) - z))
				(\omega_1^0(X_1^{-1}(s, z)) - \omega_2^0(X_2^{-1}(s, z)))
				\, dz \\
			&= \int (a_{h(x)} K(X_1(s, x) - z))
				(\omega_2^0(X_1^{-1}(s, z)) - \omega_2^0(X_2^{-1}(s, z)))
				\, dz \\
			&\qquad
				+ \int (a_{h(x)} K(X_1(s, x) - z))
				(\omega_1^0(X_1^{-1}(s, z)) - \omega_2^0(X_1^{-1}(s, z)))
				\, dz.
	\end{align*}
	Making the two changes of variables,
	$z = X_1(s, y)$ and $z = X_2(s, y)$, we can write
	\begin{align*}
		A_2^1(s, x)
			&\le \frac{1}{\zeta(x)}
				\abs{\int (a_{h(x)} K(X_1(s, x) - X_1(s, y))
					- a_{h(x)} K(X_1(s, x) - X_2(s, y))) \omega_2^0(y) \, dy} \\
			&\qquad
				+ \frac{\abs{J(s, x)}}{\zeta(x)}.
	\end{align*}
	We are thus in a position to apply \cref{P:olgKBound} to bound $A_2^1$.
	To do so, we set
	\begin{align*}
		U_j := \set{y \in \R^2 \colon \abs{X_1(s, x) - X_j(s, y)} \le h(x)}
	\end{align*}
	so that $V := U_1 \cup U_2$ is as in \cref{P:olgKBound}.
	Then with $\delta$ as in \cref{e:deltaVDef}, we have
	\begin{align}\label{e:deltasBound}
		\begin{split}
		\delta(s)
			&\le \eta(s) \sup_{y \in V} \zeta(y)
			\le \eta(s) \zeta(\abs{x} + h(x) + C t \zeta(\abs{x} + h(x))) \\
			&= C \eta(s) \zeta \pr{\abs{x} + \zeta(\abs{x})
				+ C T \zeta(\abs{x} + \zeta(\abs{x}))}
			\le C_1 \eta(s) \zeta(x),
		\end{split}
	\end{align}
	where $C_1 = C(T)$.
	Above, we applied \cref{L:FlowBounds} in the second inequality and
	the last inequality follows from repeated applications of
	\cref{L:preGB} to $\zeta$.
	Hence,
	\cref{P:olgKBound} gives
	\begin{align*}
		A_2^1(s, x)
			&\le C \smallnorm{\omega^0}_{L^\iny}
					\frac{h(x)}{\zeta(x)}
					\omu \pr{\frac{\delta(s)}{h(x)}}
					+ \frac{\abs{J(s, x)}}{\zeta(x)}.
	\end{align*}
	But by \cref{L:omuSubAdditive} (noting that $h(x)/\zeta(x) \le 1$)
	and \cref{e:deltasBound},
	\begin{align*}
		\frac{h(x)}{\zeta(x)}
					\omu \pr{\frac{\delta(s)}{h(x)}}
			& \le \omu \pr{\frac{h(x)}{\zeta(x)} \frac{\delta(s)}{h(x)}}
			= \omu \pr{\frac{\delta(s)}{\zeta(x)}}
			\le \omu(C_1 \eta(s)).
	\end{align*}
	Hence,
	\begin{align}\label{e:A21Bound}
		A_2^1(s, x)
			\le C\omu \pr{C_1 \eta(s)}
				+ \frac{\abs{J(s, x)}}{\zeta(x)}.
	\end{align}	

	We now bound $A_2^2(x)$. We have,
\begingroup
\allowdisplaybreaks
	\begin{align*}
		A_2^2(s, x)
		&\le \frac{C}{\zeta(x)} \int_0^s
			\max \bigset{\norm{\frac{u_1}{h}}_{L^\iny(0, T) \times \R^2},
			\norm{\frac{u_2}{h}}_{L^\iny(0, T) \times \R^2}} \\
			&\qquad\qquad
			\int_{B_{\frac{h(x)}{2}}(X_1(r, x))^C}
				\frac{(\zeta h)(y)}{\abs{X_1(r, x) - y}^3}
        		\frac{\abs{u_1(r, y) - u_2(r, y)}}{\zeta(y)} \, dy \, dr\\
		&\le \frac{C}{\zeta(x)}
			\int_0^s \pr{\sup_{z \in \R^2}
				\frac{\abs{u_1(r, z) - u_2(r, z)}}{\zeta(z)} 
			\int_{B_{\frac{h(x)}{2}}(X_1(r, x))^C} \frac{(\zeta h)(y)}
				{\abs{X_1(r, x) - y}^3} \, dy}
        			dr \\
		&= \frac{C}{\zeta(x)}
			\int_0^s \left(Q(r)\int_{B_{\frac{h(x)}{2}}(X_1(r, x))^C} \frac{(\zeta h)(y)}
				{\abs{X_1(r, x) - y}^3} \, dy 
			   \right)\, dr.
	\end{align*}
\endgroup
	Because $\zeta h$ is subadditive (being a \preGB),
	letting $w = X_1(r, x)$, we have
	\begin{align}\label{e:A22Split}
		\begin{split}
		\int_{B_{\frac{h(x)}{2}}(w)^C} &\frac{(\zeta h)(y)}
					{\abs{w - y}^3} \, dy
			\le \int_{B_{\frac{h(x)}{2}}(w)^C} \frac{(\zeta h)(w - y)}
					{\abs{w - y}^3} \, dy
				+ (\zeta h)(w) \int_{B_{\frac{h(x)}{2}}(w)^C} \frac{1}
					{\abs{w - y}^3} \, dy \\
			&= 2 \pi H[\zeta h](h(x)/2)
				+ C \frac{(\zeta h)(w)}{h(x)/2}
			\le C
				+ C \zeta (w)
			\le C(1 + \zeta(x)).
		\end{split}
	\end{align}
	Here we used that $\zeta h$ is a \GB and \cref{L:FlowBounds}.
	It follows that
	\begin{align}\label{e:A22EarlyBound}
		A_2^2(s, x)
		&\le C \frac{1 + \zeta(x)}{\zeta(x)} \int_0^s  Q(r) \, dr
		\le C \int_0^s Q(r) \, dr.
	\end{align}
	
	But,
	\begin{align*}
		Q(r)
			 &=\sup_{z \in \R^2}
				\frac{\abs{u_1(r, z) - u_2(r, z)}}{\zeta(z)}
			= \sup_{z \in \R^2}
				\frac{\abs{u_1(r, X_1(r, z)) - u_2(r, X_1(r, z))}}
				{\zeta(X_1(r, z))} \\
			&\le C \sup_{z \in \R^2}
				\frac{\abs{u_2(r, X_1(r, z)) - u_2(r, X_2(r, z))}}{\zeta(z)}
					+ C \sup_{z \in \R^2}
				\frac{\abs{u_2(r, X_2(r, z)) - u_1(r, X_1(r, z))}}{\zeta(z)} \\
			&\le C \pr{\omu(\eta(r)) + L(r)},
	\end{align*}
	where we used \cref{L:FlowBounds} in the first inequality
	and \cref{e:A1Bound} in the last inequality.
	Hence,
	\begin{align*}
		A_2^2(s, x)
			\le C \int_0^s (\omu(\eta(r)) + L(r)) \, dr.
\end{align*}

It follows from all of these estimates that
\begin{align}\label{e:LBoundForUniqueness}
	\begin{split}
	\eta(t)
		&\le \int_0^t L(s) \, ds
		= M(t), \\\
	L(s)
		&\le \norm{\frac{u_1^0 - u_2^0}{\zeta}}_{L^\iny}
			+ C \omu(C_1 \eta(s))
			+ \frac{\abs{J(s, x)}}{\zeta(x)}
			+ C \int_0^s (\omu(\eta(r)) + L(r)) \, dr \\
		&= a(T)
			+ C \omu(C_1 \eta(s))
			+ C \int_0^s (\omu(C_1 \eta(r)) + L(r)) \, dr.
	\end{split}
\end{align}

We therefore have
\begin{align}\label{e:MBoundForUniqueness}
	\begin{split}
	M(t)
		&\le t a(T) + \int_0^t
			\pr{C \omu(C_1 \eta(s))
			+ C \int_0^s (\omu(C_1 \eta(r)) + L(r)) \, dr}
			\, ds \\
		&\le t a(T) + C \int_0^t
			\pr{\omu(C_1 M(s))
			+ \int_0^s (\omu(C_1 M(r)) + L(r)) \, dr}
			\, ds \\
		&= t a(T) + C \int_0^t
			\pr{\omu(C_1 M(s)) + M(s)
			+ \int_0^s \omu(C_1 M(r)) \, dr}
			\, ds \\
		&\le t a(T) + C \int_0^t
			\pr{(1 + s)\omu(C_1 M(s)) + M(s)}
			\, ds \\
		&\le t a(T) + C \int_0^t
			\pr{\omu(C_1 M(s)) + M(s)} \, ds,
	\end{split}
\end{align}
where we note that the final $C = C(T)$ increases with $T$.
In the second inequality we used $\omu$ increasing and $\eta(s) \le M(s)$,
while in the third inequality we used that $\omu$ and $M$ are both increasing.
The bound in \cref{e:LProtoBoundForUniqueness} follows from Osgood's lemma.

We now obtain the bounds on $M(t)$ and $Q(t)$.

	The bound on $Q$ we made earlier shows that
	\begin{align}\label{e:QBound}
		Q(t)
			\le C(T) (\omu(\eta(t)) + L(t))
			\le C(T) (\omu(C_1 M(t)) + L(t)),
	\end{align}
	since $\eta(t) \le M(t)$. Then by \cref{e:LBoundForUniqueness},
	\begin{align*}
		L(t)
			&\le a(T) + C(T) \omu(C_1 M(t))
				+ C(T) \int_0^t (\omu(C_1 M(s)) + L(s)) \, ds \\
			&\le a(T) + C(T) \omu(C_1 M(t)))
				+ C(T) \int_0^t L(s) \, ds.
	\end{align*}
	Applying Gronwall's inequality,
	\begin{align*}
		L(t)
			\le (a(T) + C \omu(C_1 M(t))) e^{C(T) t}.
	\end{align*}
	Since we can absorb a constant, this same bound holds for $Q(t)$:
	\begin{align*}
		Q(t)
			\le (a(T) + C \omu(C_1 M(t))) e^{C(T) t}.
	\end{align*}
	Hence, we can easily translate a bound on $M$ to a bound on $Q$.
	
	Returning, then, to \cref{e:LProtoBoundForUniqueness},
	we have,
	for $a(T)$ sufficiently small,
	\begin{align*}
		\int_{t a(T)}^{M(t)} \frac{ds}{\omu(C_1 s)} \,ds
			\le C(T) t.
	\end{align*}
	Integrating gives
	\begin{align*}
		- \log\log s\big\vert_{C_1 t a(T)}^{C_1 M(t)} \le C(T) t
	\end{align*}
	from which we conclude that
	\begin{align*}
		M(t)
			\le (t a(T))^{e^{-C(T)t}}.
	\end{align*}
	This holds as long as $M(t) \le C_1^{-1} e^{-1}$, which gives a bound
	on the time $t$.
	
	On the other hand, if $s > C_1^{-1} e^{-1}$ then $\omu(C_1 s) = e^{-1}$, so
	for $ta(T) > C_1^{-1} e^{-1}$,
	\begin{align*}
		\int_{t a(T)}^{M(t)} \frac{ds}{\omu(C_1 s)} \,ds = \int_{ta(T)}^{M(t)} e  \,ds
			\le C(T) t.
	\end{align*}
	Thus,
	\begin{align*}
		e(M(t) -ta(T))\le C(T) t \leq C(T) t a(T),
	\end{align*}
	giving
	\begin{align*}
		M(t)
			\le C(T)ta(T).
	\end{align*}
	(For intermediate values of $a(T)$ we still obtain a usable
	bound, it is just more difficult to be explicit.)
\end{proof}

	
	We were able to use a \GB $\zeta$ larger than
	$h$ and obtain a result for an arbitrary $T$ because, unlike the proof
	of existence in \cref{S:Existence}, we are assuming that
	we already know that $u_1, u_2$ lies in $S_h$. Hence, the quadratic
	term in the Serfati identity can in effect be made linear.

\Ignore{ 
	The terms $A_1$ and $A_2^2$ in the proof above of \cref{T:ProtoUniqueness}
	are the controlling terms in the sense that they are the largest.
	When $u_1^0 = u_2^0$, so that $J \equiv 0$, the only term involving
	the vorticity
	is $A_2^1$, which we note is of order $h(x)/\zeta(x) \le 1$.
	We could allow the vorticity
	to grow at infinity at the rate $\xi(x)$ for a \GB $\xi$ for which
	$\xi \le C h$ and still obtain a bound on
	$A_2^1$ of the same order as $A_1$ and $A_2^2$. If we assume as well
	that $h \xi \le C \zeta$ then the bound on $A_2^2$ can be recovered
	as well.
	The bound on $A_1$,
	however, would increase, as an examination of the proof of \cref{P:Morrey}
	shows that a factor of $\xi(x)$ would be introduced,
	a factor that cannot be compensated for in $A_2^2$.
		
	Finally, in bounding $A_2^2$ we used the subadditivity of $\zeta h$
	in the form $(\zeta h)(y) \le (\zeta h)(w - y) + (\zeta h)(w)$.
	When $w$ is close to $y$, this bound becomes close
	to equality with the $(\zeta h)(w)$ term dominating. And because $h(w) << w$
	for $w$ large in integrating over
	$B_{h(x)}(w)^C$, $y$ is close to $w$ for the largest values of the integrand.
	Hence, it is the second integral in the estimate in \cref{e:A22Split}
	that dominates and we conclude that this bound is close to equality
	and cannot effectively be improved.
	} 

\cref{T:ProtoUniqueness} gives a bound on the difference in velocities over time. It remains, however, to characterize $a(T)$ in a useful way in terms of $u_1^0$, $u_2^0$, and $u_1^0 - u_2^0$ and so obtain \cref{T:aT}. This, the subject of the next section, is not as simple as it may seem.

\Ignore{

\ToDo{I think it is worth keeping this lemma and referring to it tangentially at some point, though we can't use it. The point is that the Serfati identity implies uniform convergence of the renormalized Biot-Savart law, but the reverse implication I don't think need hold when $h$ is not constant.}
\begin{lemma}\label{L:RenBSBound}
	Let $h$ be a \GBU and let $u$ be a solution
	in $S_h$ on $[0, T]$ with $u(0) = u^0 \in S_h$.
	Then
	\begin{align*}
		u(t) - u^0
			&= \lim_{R \to \iny}
				(a_R K)*(\omega(t) - \omega^0),
	\end{align*}
	the convergence being uniform over $[0, T] \times \R^2$.
\end{lemma}
\begin{proof}
	As in \cref{e:VelTermBound}, we have
	\begin{align*}
		&\abs{u(t, x) - u^0(x) - (a_R K)*(\omega(t) - \omega^0(t))(x)} \\
			&\qquad
			\le \int_0^t \abs{\pr{\grad \grad^\perp \brac{(1 - a_\la) K}}
        			\stardot (u_1 \otimes u_1 - u_2 \otimes u_2)(s, x)} \, ds \\
			&\qquad
			\le C \int_0^t \sum_j \norm{g u_j(s)}_{L^\iny}^2
				\brac{H(\la(x)) + C \frac{h(x)^2}{\la(x)}}.
	\end{align*}
	Choosing $\la(x) = R h(x)^2$, we have (using $H$ decreasing and $h(x) \ge 1$)
	\begin{align*}
		&\smallnorm{u(t) - u^0 - (a_R K)*(\omega(t) - \omega^0)}_{L^\iny(\R^2)}
			\le C T \brac{H(R) + R^{-1}}.
	\end{align*}
	This gives the result, since $H(R) \to 0$ as $R \to \iny$.
\end{proof}
} 

%
%
\section{Continuous dependence on initial data}\label{S:ContDep}

\noindent In this section, we prove \cref{T:aT,T:aTSimple}, bounding $a(T)$ of \cref{e:aT}. The difficulty in bounding $a(T)$ lies wholly in bounding $J\slash \zeta$, with $J = J(t, x)$ as in \cref{e:J}.  We can write $J = J_2 - J_1$, where
\begin{align*}
	J_1(t, x)
		&= (a_{h(x)} K) * (\omega_1^0 - \omega_2^0)(X_1(t, x)), \\
	J_2(t, x)
		&= (a_{h(x)} K) * ((\omega_1^0 - \omega_2^0) \circ (X_1^{-1}(t))(x).
\end{align*}

Both $J_1/\zeta$ and $J_2/\zeta$ are easy to bound, as we do in \cref{T:aTSimple}, if we assume that $\omega_1^0 - \omega_2^0$ is close in $L^\iny$, an assumption that is physically unreasonable, however, as discussed in \cref{S:Introduction}.

\begin{proof}[\textbf{Proof of \cref{T:aTSimple}}]
	We have
	\begin{align*}
		\abs{\frac{J_1(s, x)}{\zeta(x)}}
			&\le \norm{(a_{h(x)} K) *
				(\omega_1^0 - \omega_2^0)(X_1(s, z))}_{L^\iny_z}/\zeta(x) \\
			&= \norm{(a_{h(x)} K) *
				(\omega_1^0 - \omega_2^0)}_{L^\iny}/\zeta(x) \\
			&\le \norm{a_{h(x)} K}_{L^1}
				\norm{\omega_1^0 - \omega_2^0}_{L^\iny}/\zeta(x)
			\le C (h(x)/\zeta(x)) \norm{\omega_1^0 - \omega_2^0}_{L^\iny} \\
			&\le C \norm{\omega_1^0 - \omega_2^0}_{L^\iny}.
	\end{align*}
	Similarly,
	\begin{align*}
		\abs{\frac{J_2(s, x)}{\zeta(x)}}
			&\le \norm{(a_{h(x)} K) *
				((\omega_1^0 - \omega_2^0)
					\circ X_1^{-1}(s))(x)}_{L^\iny_z}/\zeta(x) \\
			&\le \norm{a_{h(x)} K}_{L^1}
				\norm{(\omega_1^0 - \omega_2^0)(X_1^{-1}(s, z))}
					_{L^\iny_z}/\zeta(x)
			\le C (h(x)/\zeta(x)) \norm{\omega_1^0 - \omega_2^0}_{L^\iny} \\
			&\le C \norm{\omega_1^0 - \omega_2^0}_{L^\iny}.
	\end{align*}
	Combined, these two bounds easily yield the bound on $a(T)$.
\end{proof}

More interesting is a measure of $a(T)$ in terms of $u_1^0 - u_2^0$ without involving $\omega_1^0 - \omega_2^0$. Now, $J_1$ is fairly easily bounded in terms of $u_1^0 - u_2^0$ (using \cref{L:alaKZBound}) since $X_1(s, x)$ has no effect on the $L^\iny$ norm. But in $J_2$, the composition of the initial vorticity with the flow map complicates matters considerably. What we seek is a bound on $a(T)$ of \cref{e:aT} in terms of $\smallnorm{(u_1^0 - u_2^0)/\zeta}_{L^\iny}$ and constants that depend upon $\smallnorm{u_1^0}_{S_1}$, $\smallnorm{u_2^0}_{S_1}$. That is the primary purpose of \cref{T:aT}, which we now prove, making forward references to a number of results that appear following its proof.  These include \cref{L:ForJ1Bound,P:f0XBound,L:HolderInterpolation}, which employ Littlewood-Paley theory and \Holder spaces of negative index, and which we defer to \cref{S:LP}, where we introduce the necessary technology.

\Ignore{ 
\begin{definition}\label{D:Holder}
	For any $r = k + \al$
	for $k$ a nonnegative integer with $\al \in (0, 1)$, we define
	the \Holder space, $C^r(\R^2) = C^{k, \al}(\R^2)$, to be the space
	of all $k$-times continuously differentiable functions such that
	\begin{align*}
		\norm{f}_{C^r}
			= \norm{f}_{C^{k, \al}}
			:= \sum_{\abs{\gamma} \le k}
				\pr{\norm{D^\gamma f}_{L^\iny}
					+ \sup_{x \ne y}
						\frac{\abs{D^\gamma f(x) - D^\gamma f(y)}}
						{\abs{x - y}^\al}
				},
	\end{align*}
	where the sum is over multi-indices, $\gamma$. (Of course,
	$C^k(\R^2)$ is defined
	the same way, without the $\sup_{x \ne y}$ term in the norm.)
\end{definition}
} 

\begin{remark}\label{R:NegHolder}
	In the proof of \cref{T:aT}, we make use for the first time in this paper
	of \Holder spaces, with negative and fractional indices.
	We are not using the classical
	definition of these spaces, but rather one based upon Littlewood-Paley theory.
	For non-integer indices, they are equivalent, but the constant of
	equivalency (in one direction) blows up as the index approaches an
	integer (see \cref{R:HolderWarning}).
	Because we will be comparing norms with different indices,
	it is important that we use a consistent definition of these spaces.
	In this section, the only fact we use regarding \Holder spaces
	(in the proof of \cref{L:SmallVelSmallVorticity})
	is that $\norm{\dv v}_{C^{r - 1}} \le C \norm{ v}_{C^r}$
	for any $r \in (0, 1)$ for a constant $C$ independent of $r$.
	For that reason, we defer our definition
	of \Holder spaces to \cref{S:LP}.
\end{remark}

\begin{remark}\label{R:TechnicalIssue}
	With the exception of \cref{P:f0XBound}, versions of all of the
	various lemmas and propositions that we use in the proof
	of \cref{T:aT} can be obtained for solutions in $S_h$
	for any \GBU $h$.
	Should a way be found to also extend \cref{P:f0XBound} to
	$S_h$ then a version of \cref{T:aT}
	would hold for $S_h$ as well.
\end{remark}

\begin{proof}[\textbf{Proof of \cref{T:aT}}]
	Let $\omega^0_1$, $\omega^0_2$
	be the initial vorticities,
	and let $\omega_1, \omega_2$
	and $X_1, X_2$ be the vorticities and flow maps
	of $u_1$, $u_2$.

	To bound $a(T)$, let
	$\ol{\omega}_0 = \omega_1^0 - \omega_2^0 = \curl (u_1^0 - u_2^0)$.
	Then, since $h \equiv 1$, we can write,
	\begin{align*}
		\frac{J_1(s, x)}{\zeta(x)}
			&=  \frac{\zeta(X_1(s, x))}{\zeta(x)}
				\frac{\brac{(a K) * \curl (u_1^0 - u_2^0)}(X_1(s, x))}
				{\zeta(X_1(s, x))}, \\
		\frac{J_2(s, x)}{\zeta(x)}
			&= \frac{(a K) *
				(\ol{\omega}_0 \circ X_1^{-1}(s))(x)}{\zeta(x)}.
	\end{align*}
	Applying \cref{L:preGB}, \cref{L:FlowBounds},
	and \cref{L:alaKZBound}, we see that
	\begin{align}\label{e:J1SecondaBound}
		\frac{J_1(s, x)}{\zeta(x)}
			&\le
			C \norm{\frac{u_1^0 - u_2^0}{\zeta}}_{L^\iny},
	\end{align}
	the bound holding uniformly over $s \in [0, T]$.
	
	Bounding $J_2$ is much more difficult, because the flow map appears
	inside the convolution, which prevents us from writing it as the
	curl of a divergence-free
	vector field. Instead, we apply a sequence of bounds,
	starting with
	\begin{align*}
		\abs{\frac{J_2(s, x)}{\zeta(x)}}
			&=
				\abs{\frac{(\zeta \circ X_1^{-1}(s))(x)}{\zeta(x)}}
				\abs{\frac{(a K) *
				(\ol{\omega}_0 \circ X_1^{-1}(s))(x)}
				{(\zeta \circ X_1^{-1}(s))(x)}} \\
			&\le C
				\abs{\frac{(a K) *
				(\ol{\omega}_0 \circ X_1^{-1}(s))(x)}
				{(\zeta \circ X_1^{-1}(s))(x)}} \\
			&\le C \Phi_\al \pr{s, 
				\norm{\frac{\ol{\omega}_0}{\zeta} \circ X_1^{-1}(s, \cdot)}
					_{C^{-\al}}
				}
	\end{align*}
	for all $s \in [0, T]$.
	Here we used \cref{L:FlowBounds} and
	applied \cref{L:ForJ1Bound} with $f = \ol{\omega}_0 \circ X_1^{-1}(s)$.
	
	Applying \cref{P:f0XBound} with $\alpha = \delta_t$
	and $\beta = 2 C(\delta)$,
	followed by \cref{L:SmallVelSmallVorticity} gives	
	\begin{align*}
		\abs{\frac{J_2(s, x)}{\zeta(x)}}
			&\le C \Phi_\al \pr{s,
				2 \norm{\frac{\ol{\omega}_0}{\zeta}}_{C^{-\delta}}
				}
			\le
				C \Phi_\al \pr{s,
					2
					\norm{\frac{u_1^0 - u_2^0}{\zeta}}_{C^{1 - \delta}}
				}
	\end{align*}
	for all $s \in [0, T^*]$.
	Note that the condition on $\delta_{T^*}$
	in \cref{P:f0XBound} is satisfied
	because of our definition of $T^*$
	and because $\norm{u}_{LL} \le C \norm{u}_{S_1}$, which
	follows from \cref{P:Morrey}.
	We also used, and use again below, that $\Phi_\al$ is increasing
	in its second argument.

	We apply \cref{L:HolderInterpolation} with $r = 1 - \delta$, obtaining 
	\begin{align*}
		\norm{\frac{u_1^0 - u_2^0}{\zeta}}_{C^{1 - \delta}}
			&C \le \norm{\frac{u_1^0 - u_2^0}{\zeta}}_{L^\iny}^\delta
			\brac{\norm{\frac{u_1^0 - u_2^0}{\zeta}}_{L^\iny}^{1 - \delta}
				+ \norm{\frac{u_1^0 - u_2^0}{\zeta}}_{S_1}^{1 - \delta}
			}.
	\end{align*}
	But,
	\begin{align*}
		\norm{\frac{u_1^0 - u_2^0}{\zeta}}_{S_1}
			&\le \norm{\frac{u_1^0 - u_2^0}{\zeta}}_{L^{\infty}}
				+ \norm{\frac{1}{\zeta}}_{L^\iny}
					\norm{\omega_1^0 - \omega_2^0}_{L^\iny}
				+ \norm{\grad\pr{\frac{1}{\zeta}}}_{L^\iny}
					\norm{u_1^0 - u_2^0}_{L^\iny} \\
			&\le C \norm{u_1^0 - u_2^0}_{S_1}
			\le C,
	\end{align*}
	where we used the identity,
	$\curl (f u) = f \curl u - \grad f \cdot u^\perp$,
	and that $1/\zeta$ is Lipschitz (though $1/\zeta \notin C^1(\R^2)$
	unless $\zeta$ is constant, because
	$\grad \zeta$ is not defined at the origin).
	Therefore,
	\begin{align*}
		\norm{\frac{u_1^0 - u_2^0}{\zeta}}_{C^{1 - \delta}}
			&\le C \norm{\frac{u_1^0 - u_2^0}{\zeta}}_{L^\iny}^\delta.
	\end{align*}

	We conclude that
	\begin{align*}
		\abs{\frac{J_2(s, x)}{\zeta(x)}}
			\le
				C_1 \Phi_\al \pr{T,
					C
					\norm{\frac{u_1^0 - u_2^0}{\zeta}}_{L^\iny}^\delta
					}
	\end{align*}
	for all $0 \le s \le T^*$.
	Since the bound on $J_1/\zeta$ in \cref{e:J1SecondaBound} is
	better than that on $J_2/\zeta$, this completes the proof.
\end{proof}

\begin{remark}
%
In the application of \cref{P:f0XBound} and \cref{L:zetaFlowBound} (which was used in the proof of \cref{L:ForJ1Bound}) the value of $C_0 = \norm{u_1}_{L^\iny(0, T; S_1)}$ enters into the constants. A bound on $\norm{u_1}_{L^\iny(0, T; S_1)}$ comes from the proof of existence in \cref{S:Existence}. While we did not explicitly calculate it, for $S_1$ it yields an exponential-in-time bound, as in \cite{Serfati1995A,AKLN2015}. Hence, our bound on $a(T)$ is doubly exponential (it would be worse for unbounded velocities). 
It is shown in \cite{Gallay2014} (extending \cite{Zelik2013}) for bounded velocity, however, that $\norm{u_1}_{L^\iny(0, T; S_1)}$ can be bounded linearly in time, which means that $C_0$ actually only increases singly exponentially in time.

Whether an improved bound can be obtained for a more general $h$ is an open question: If it could, it would extend the time of existence of solutions, possibly expanding the class of \GBs for which global-in-time existence holds.
\end{remark}

\begin{lemma}\label{L:CurlConv}
	Let $Z \in S_1$.
	For any $\la > 0$,
	\begin{align}\label{e:CurlConv}
		(a_\la K) * \curl Z
			= \curl (a_\la K) * Z.
	\end{align}
\end{lemma}
\begin{proof}
	Note that
	$
		((a_\la K) * \curl Z)^i
			= (a_\la K^i) * (\prt_1 Z^2 - \prt_2 Z^1)
			= (\prt_1 (a_\la K^i)) * Z^2 - (\prt_2 (a_\la K^i)) * Z^1
	$.
	Thus, \cref{e:CurlConv} is not just a matter of moving the curl from
	one side of the convolution to the other.
	Using both that $Z$ is divergence-free and that $a$ is radially
	symmetric, however, \cref{e:CurlConv} is
	proved in Lemma 4.4 of \cite{KBounded2015}.
\end{proof}

The following is a twist on Proposition 4.6 of \cite{BK2015}.
\begin{lemma}\label{L:alaKZBound}
	Let $\zeta$ be a \preGB
	and suppose that $Z \in S_1$.
    For any $\la > 0$,
    \begin{align}\label{e:alaKZBound}
    	\begin{split}
			\norm{(a_\la K) * \curl Z}_{L^\iny(\R^2)}
				&\le 2 \norm{Z}_{L^\iny(\R^2)}, \\
			\norm{\frac{(a_\la K) * \curl Z}{\zeta}}_{L^\iny(\R^2)}
				&\le \pr{ 1 +6 \frac{\zeta(\la)}{\zeta(0)}}
					\norm{\frac{Z}{\zeta}}_{L^\iny(\R^2)}.
		\end{split}
    \end{align}
\end{lemma}
\begin{proof}
\cref{L:CurlConv} gives
$	
	(a_\la K) * \curl Z
		= \curl (a_\la K) * Z
$.
But $\curl (a_\la K) = - \dv(a_\la K^\perp) = - a_\la \dv K^\perp - \grad a_\la \cdot K^\perp = \delta - \grad a_\la \cdot K^\perp$, where $\delta$ is the Dirac delta function, since $a_\la(0) = 1$. Hence,
\begin{align}\label{e:KernelCurlCommute}
	(a_\la K) * \curl Z
		&= Z - (\grad a_\la \cdot K^\perp)* Z
		= Z - \varphi_\la* Z,
\end{align}
where $\varphi_\la := \grad a_\la \cdot K^\perp \in C_c^\iny(\R^2)$. Then \cref{e:alaKZBound}$_1$ follows from
\begin{align*}
	\norm{(a_\la K) * \curl Z}_{L^\iny(\R^2)}
		&\le C \pr{1 + \norm{\varphi_\la}_{L^1}} \norm{Z}_{L^\iny(\R^2)}
		= 2 \norm{Z}_{L^\iny(\R^2)}.
\end{align*}
Here, we have used that $\norm{\varphi_\la}_{L^1} = 1$, as can easily be verified by integrating by parts. (In fact, $\varphi_\la *$ is a mollifier, though we will not need that.)

Using \cref{e:KernelCurlCommute}, we have
\begin{align*}
	\norm{\frac{(a_\la K) * \curl Z}{\zeta}}_{L^\iny}
		&\le \norm{\frac{Z}{\zeta}}_{L^\iny(\R^2)}
			+ \norm{\frac{\varphi_\la * Z}{\zeta}}_{L^\iny(\R^2)}.
\end{align*}
But $\varphi_\la(x-\cdot)$ is supported in $B_\la(x)$, so
\begin{align*}
	\abs{\frac{\varphi_\la * Z(x)}{\zeta(x)}}
		&= \abs{\frac{\varphi_\la *
			(\CharFunc_{B_\la(x)} Z)(x)}{\zeta(x)}}
		\le \frac{\norm{\varphi_\la}_{L^1}
			\norm{\CharFunc_{B_\la(x)} Z}_{L^\iny}}{\zeta(x)}
		= \frac{\norm{\CharFunc_{B_\la(x)} Z}_{L^\iny}}{\zeta(x)} \\
		&\le \norm{\frac{Z}{\zeta}}_{L^\iny}
			\frac{\zeta(\abs{x} + \la)}
			{\zeta \pr{\max \set{\abs{x} - \la, 0}}}.
\end{align*}
Taking the supremum over $x \in \R^2$, we have
\begin{align*}
	\norm{\frac{(a_\la K) * \curl Z}{\zeta}}_{L^\iny(\R^2)}
		&\le \norm{\frac{Z}{\zeta}}_{L^\iny}\pr{ 1 +
			\sup_{x \in \R^2}
				\frac{\zeta(\abs{x} + \la)}
				{\zeta\pr{\max \set{\abs{x} - \la, 0}}} }.
\end{align*}

If $\abs{x} > 2 \la$ then using \cref{L:preGB},
\begin{align*}
	\frac{\zeta(\abs{x} + \la)}{\zeta\pr{\max \set{\abs{x} - \la, 0}}}
		\le 	\frac{\zeta(3 \abs{x}/2)}{\zeta(\abs{x}/2)}
		\le 6 \frac{\zeta(\abs{x}/2)}{\zeta(\abs{x}/2)}
		= 6,
\end{align*}
while if $\abs{x} \le 2 \la$ then
\begin{align*}
	\frac{\zeta(\abs{x} + \la)}{\zeta\pr{\max \set{\abs{x} - \la, 0}}}
		\le \frac{\zeta(3 \la)}{\zeta(0)}
		\le 6 \frac{\zeta(\la)}{\zeta(0)}.
\end{align*}
Hence, we obtain \cref{e:alaKZBound}$_2$.
\end{proof}

\begin{lemma}\label{L:SmallVelSmallVorticity}
	Let $\zeta$ be a \preGB, $\al \in (0, 1)$, and
	$\ol{\omega}_0 = \omega_1^0 - \omega_2^0 = \curl(u_1^0 - u_2^0)$
	with $u_1^0$, $u_2^0 \in S_\zeta$. Then
	\begin{align*}
		\norm{\frac{\ol{\omega}_0}{\zeta}}_{C^{\al - 1}}
			\le C \norm{\frac{u_1^0 - u_2^0}{\zeta}}_{C^\al}.
	\end{align*}
\end{lemma}
\begin{proof}
	Let $g := 1/\zeta$ and
	$v = -(u_1^0 - u_2^0)$ so that $\ol{\omega}_0 = \dv v^\perp$.
	Then
	\begin{align*}
		g \ol{\omega}_0
			= g \dv v^\perp 
			= \dv (g v^\perp) - \grad g \cdot v^\perp.
	\end{align*}
	Thus (see \cref{R:NegHolder}),
	\begin{align*}
		&\norm{g \ol{\omega}_0}_{C^{\al - 1}}
			\le C \pr{\smallnorm{g v^\perp}_{C^\al}
				+ \smallnorm{\grad g \cdot v^\perp}_{C^{\al-1}}}
			= C \pr{\norm{g v}_{C^\al}
				+ \smallnorm{\grad g \cdot v^\perp}_{C^{\al-1}}}.
	\end{align*} 
	But by virtue of \cref{L:fhzetaacts}, we also have
	\begin{align*}
		\smallnorm{\grad g \cdot v^\perp}_{C^{\al-1}}
			\le \smallnorm{\grad g \cdot v^\perp}_{L^{\infty}}
			\le c_0 \smallnorm{g v}_{L^{\infty}}
			\le c_0\smallnorm{g v}_{C^{\al}}.
	\end{align*}
\end{proof}

\Ignore{ 
\begin{lemma}\label{L:PhiSubmultiplicative}
	Let $\al, \beta > 0$ and define $\Phi_\al$
	as in \cref{e:Phialpha}.
	For all $\al > 0$ and all $a, x \in [0, \iny)$,
	\begin{align*}
		\Phi_\al(t, ax) \le \Phi_\al(t, a) \Phi_\al(t, x).
	\end{align*}
\end{lemma}
\begin{proof}
	Let $\beta_t := \frac{e^{-C_0t}}{\al + e^{-C_0t}}$ so that
	$\Phi_\al(t, x) := x + x^{\beta_t}$. Then
	\begin{align*}
		\Phi_\al(t, ax)
			= ax + (ax)^{\beta_t}
			\le (a + a^{\beta_t}) (x + x^{\beta_t})
			= \Phi_\al(t, a) \Phi_\al(t, x).
	\end{align*}
\end{proof}
} 


\Ignore{ 
\ToDo{Maybe turn what follows, shortened, into a comment somewhere, as I have at least twice now fooled myself into thinking it should work and tremendously simplify things. Point out why it does not.}

	We decompose $A_2^1(s, x)$ as
	\begin{align*}
		A_2^1(s, x)
			&= \abs{\frac{J_2(s, x)}{\zeta(x)} - \frac{J_1(s, x)}{\zeta(x)}},
	\end{align*}
	where
	\begin{align*}
		J_2(s, x)
			&= 
				\frac{\brac{(a_{h(x)} K) * (\omega_1(s) - \omega_2(s))}(X_1(s, x))}
				{\zeta(x)}, \\
		J_1(s, x)
			&= \frac{\brac{(a_{h(x)} K) * (\omega_1^0 - \omega_2^0)}(X_1(s, x))}
				{\zeta(x)}.
	\end{align*}
	But, $\omega_1(s) - \omega_2(s) = \curl (u_1(s) - u_2(s)) = \dv v(s)$,
	where $v = (u_1 - u_2)^\perp$, so, also setting $v^0 = v(0)$,
	\begin{align*}
		J_2(s, x)
			&= 
				\frac{\zeta(X_1(s, x))}{\zeta(x)}
				\frac{\brac{(a_{h(x)} K) * \dv v(s)}(X_1(s, x))}
				{\zeta(X_1(s, x))}, \\
		J_1(s, x)
			&= \frac{\zeta(X_1(s, x))}{\zeta(x)}
				\frac{\brac{(a_{h(x)} K) * \dv v^0}(X_1(s, x))}
				{\zeta(X_1(s, x))}.
	\end{align*}
	Applying \cref{L:preGB}, \cref{L:FlowBounds},
	and \cref{L:alaKZBound}, we see that
	\begin{align*}
		\abs{J_2(s,  x)}
			&\le
			C \norm{\frac{u_1(s) - u_2(s)}{\zeta}}_{L^\iny}, \\
		\abs{J_1(s, x)}
			&\le
			C \norm{\frac{u_1^0 - u_2^0}{\zeta}}_{L^\iny}
	\end{align*}
	for all $s \in [0, T]$.
} 

\section{\Holder space estimates}\label{S:LP}

In this section, we make use of the Littlewood-Paley operators $\Delta_j$, $j\geq -1$.  A detailed definition of these operators and their properties can be found in chapter 2 of \cite{C1998}. We note here, only that $\Delta_j f = \varphi_j * f$, where $\varphi_j(\cdot) = 2^{2j} \varphi(2^j \cdot)$ for $j \ge 0$, $\varphi$ is a Schwartz function, and the Fourier transform of $\varphi$ is
supported in an annulus. We can  write $\Delta_{-1} f$ as a convolution with a Schwartz function $\chi$ whose Fourier transform is supported in a ball.

We will also make use of the following Littlewood-Paley definition of Holder spaces.
\begin{definition}\label{D:LPHolder}
Let $r\in\R$.  The space $C_*^r(\R^2)$ is defined to be the set of all tempered distributions on $\R^2$ for which 
\begin{equation*}
\| f \|_{C_*^r} := 2^{r j}
	\sup_{j\geq -1} \|\Delta_j f \|_{L^{\infty}} <\infty. 
\end{equation*}
\end{definition}

\begin{remark}\label{R:HolderWarning}
It follows from Propositions 2.3.1, 2.3.2 of \cite{C1998} that the $C_*^r$ norm is equivalent to the classical \Holder space $C^r$ norm when $r$ is a positive non-integer: $\norm{f}_{C^r} \le a_r \norm{f}_{C_*^r}$,  $\norm{f}_{C_*^r} \le b_r \norm{f}_{C^r}$ for constants, $a_r$, $b_r > 0$, though $a_r \to \iny$ as $r$ approaches an integer. See \cref{R:NegHolder}.
\end{remark}

\Ignore{ 
\begin{lemma}\label{L:HolderEquivalence}
	There exists a constant $C > 0$ such that for any positive non-integer
	$r$, which we write as $r = k + \al$ for $k$ a integer, $\al \in (0, 1)$,
	we have
	\begin{align*}
		C \al \norm{f}_{C^r}
			\le \norm{f}_{C_*^r}
			\le \frac{C^{r + 1}}{k!} \norm{f}_{C^r}.
	\end{align*}
\end{lemma}
\begin{proof}
	The follows from 
	Proposition2 2.3.1, 2.3.2 of \cite{C1998}.
\end{proof}

\cref{L:HolderEquivalence} shows that the $C_*^r$ norm is equivalent to the classical Holder space $C^r$ norm when $r$ is a positive non-integer. As $r$ decreases to an integer, however, the $C^r$ norm may blow up while the $C_*^r$ norm remains finite, and controlling the $C^r_*$ norm gives less and less control over the $C^r$ norm.
} 

\begin{prop}\label{L:ForJ1Bound}
Let $\zeta$ be a \preGB and let $u \in L^{\infty}(0, T; S_1)$ with $X$ its associated flow map.  Let $t \in [0, T]$ and set $\eta=\zeta\circ X^{-1}(t)$.  For any $\alpha > 0$, $\la>0$, and $f \in L^\iny(\R^2)$,
\begin{align}\label{e:TheJ1Bound}
	\norm{\frac{(a_{\la}K) * f}{\eta}}_{L^\iny}
		\le C(1 + \norm{f}_{L^\iny})
			\Phi_\al \pr{t, \norm{\frac{f}{\eta}}_{C^{-\al}}},
\end{align}
	where
	$C = C(T, \zeta, \la)$,
	and $\Phi_\al$ is defined in
	\cref{e:Phialpha} (using $C_0 = \norm{u}_{L^\iny(0, T; S_1)}$).
\end{prop}
\begin{proof}
Define $g = 1/\eta$. 
For fixed $N \ge -1$ (to be chosen later), write
\begin{equation}\label{e:initial}
\begin{split}
&\abs{\frac{(a_{\la}K) * f(x)}{\eta(x)}}
	= \left| g(x) \int_{\R^2} a_{\la}(y) K(y) \eta(x - y)(f/\eta)(x - y)\, dy  \right|\\
&\qquad  = \left | g(x) \int_{\R^2} a_{\la}(y) K(y)\eta(x - y)\sum_{j\geq -1}(\Delta_j(f/\eta))(x - y)\, dy  \right|\\
&\qquad \leq \sum_{-1\leq j\leq N} \left | g(x) \int_{\R^2} a_{\la}(y) K(y)\eta(x - y)(\Delta_j(f/\eta))(x - y)\, dy  \right| \\
&\qquad + \sum_{j > N} \left | g(x) \int_{\R^2} a_{\la}(y) K(y)\eta(x - y)(\Delta_j(f/\eta))(x - y)\, dy  \right| =: I + II.
\end{split}
\end{equation}
We first estimate $I$.  Exploiting \cref{D:LPHolder},  
\begin{equation}\label{e:I}
\begin{split}
&I =  \sum_{-1\leq j\leq N} 2^{-j \alpha} 2^{j \alpha} \left | g(x) \int_{\R^2} a_{\la}(y) K(y)\eta(x - y)(\Delta_j(f/\eta))(x - y)\, dy  \right|\\
&\qquad \leq  \sup_{j} 2^{-j\alpha} \| \Delta_j(f/\eta) \|_{L^{\infty}}\sum_{-1\leq j\leq N} 2^{j\alpha}  g(x) \int_{\R^2} |a_{\la}(y) K(y)\eta(x - y)| \, dy\\
&\qquad \leq C2^{\al N} \| f/\eta \|_{C^{-\alpha}}g(x) \int_{\R^2} |a_{\la}(y) K(y)\eta(x - y)| \, dy.
\end{split}
\end{equation}

Since $\eta=\zeta\circ X^{-1}(t)$, where $\zeta$ is a \preGB , \cref{L:zetaFlowBound} implies that
\begin{equation*}
\begin{split}
&g(x) \int_{\R^2} |a_{\la}(y) K(y)\eta(x - y)| \, dy 
\le Cg(x) \int_{\R^2} |a_{\la}(y) K(y)\zeta(x - y)| \, dy\\
&\qquad \le C\int_{\R^2} | a_{\la}(y) K(y)g(x)(\zeta(x) + \zeta(y))| \, dy \\
&\qquad \leq C\int_{\R^2} | a_{\la}(y) K(y)|(1+ g(x)\zeta(y)) \, dy
		\le C \int_0^{\la} (1 + g(x)\zeta(r)) \, dr \\
	&\qquad
		\le C\la \pr{1 + g(x)\zeta(\la)}
		\le C\la(1 + \zeta(\la))
		= C(\la),
\end{split}
\end{equation*}
where we used boundedness of $g$. Substituting this estimate into \cref{e:I}, we conclude that 
\begin{equation}\label{Ie:finalest}
I \leq C 2^{\al N} \| f/\eta \|_{C^{-\alpha}}.
\end{equation}

We now estimate $II$ by introducing a commutator and utilizing the Holder continuity of $\eta$, writing    
\begin{equation}\label{e:II}
\begin{split}
II 
&\leq \sum_{j > N} \left | g(x) \int_{\R^2} a_{\la}(y) K(y)\Delta_j f(x - y)\, dy  \right| \\
&\qquad + \sum_{j > N} \left | g(x) \int_{\R^2} a_{\la}(y) K(y)[\Delta_j f(x - y) - \eta(x - y)(\Delta_j(f/\eta))(x - y)]\, dy  \right| \\
&\qquad =: III + IV.
\end{split}
\end{equation}

We rewrite $III$ as a convolution, noting that the Littlewood-Paley operators commute with convolutions, and apply Bernstein's Lemma
and \cref{L:ALPBound} to give
\begin{equation*}
\begin{split}
	III
		&= \sum_{j > N} \abs{ g(x) ((a_{\la} K) * \Delta_j f(x))}
		= g(x) \sum_{j > N} \abs{\Delta_j(a_{\la} K * f)(x)} \\
		&\leq g(x) \sum_{j > N} 2^{-j}  \norm{\nabla \Delta_j(a_{\la} K * f)}
			_{L^{\infty}}
		\le C g(x) \norm{f}_{L^\iny} \sum_{j > N} 2^{-j}
		\le C g(x) 2^{-N} \norm{f}_{L^\iny}.
\end{split}
\end{equation*} 

To estimate $IV$, we apply \cref{L:CommutatorEstimate} and (\ref{e:aKBound}) to obtain
\begin{align*} 
\begin{split}
IV 
&\leq \sum_{j > N} g(x) \int_{\R^2} \left | a_{\la}(y) K(y)  \right | \left |  \Delta_j f(x - y) - \eta(x - y)(\Delta_j(f/\eta))(x - y) \right | \, dy \\
&\le \sum_{j > N} C \norm{f}_{L^\iny}
	2^{-je^{-C_0t}}  g(x) \int_{\R^2} \left | a_{\la}(y) K(y)  \right | \, dy
\leq C \norm{f}_{L^\iny} 2^{-Ne^{-C_0t}}
	g(x) \la.
\end{split}
\end{align*}

Substituting the estimates for $III$ and $IV$ into \cref{e:II} gives
\begin{equation*}
	II
		\le C (1 + \la) \norm{f}_{L^\iny} 2^{-Ne^{-C_0t}} g(x).
\end{equation*}   
Finally, substituting the estimates for $I$ and $II$ into \cref{e:initial} yields 
\begin{align}\label{e:finalest}
	\begin{split}
	\abs{\frac{(a_{\la}K) * f(x)}{\eta(x)}}
		&\le C \pr{1 + \norm{f}_{L^\iny}}
			\pr{2^{\al N} \norm{f/\eta}_{C^{-\al}}
			+ 2^{-Ne^{-C_0t}} g(x)},
	\end{split}
\end{align}
where $C$ depends on $\la$.
Now, we are free to choose the integer $N \ge -1$ any way we wish, but if we choose $N$ as close to 
\begin{align*}
	N^*
		&:= \frac{1}{\al + e^{-C_0t}}
			\log_2 \pr{\frac{1}{\norm{f/\eta}_{C^{-\al}}}}
		= - \frac{1}{\al + e^{-C_0t}}
			\log_2 \norm{f/\eta}_{C^{-\al}}
\end{align*}
as possible, we will be near the minimizer of the bound in \cref{e:finalest}, as long as $N^* \ge -1$ (because such an $N^*$ balances the two terms).
Calculating with $N = N^*$ gives
\begin{align*}
	2^{\al N} \norm{f/\eta}_{C^{-\al}}
		&= 2^{\frac{\al}{\al + e^{-C_0t}}
			\log_2 \pr{\frac{1}{\norm{f/\eta}_{C^{-\al}}}}}
				\norm{f/\eta}_{C^{-\al}}
		= \norm{f/\eta}_{C^{-\al}}^{\frac{e^{-C_0t}}{\al + e^{-C_0t}}},
		\\
	2^{-Ne^{-C_0t}} g(x)
		&= 2^{\frac{e^{-C_0t}}{\al + e^{-C_0t}}
			\log_2 \norm{f/\eta}_{C^{-\al}}} g(x)
		\le C\norm{f/\eta}_{C^{-\al}}^{\frac{e^{-C_0t}}{\al + e^{-C_0t}}},
\end{align*}
since $g$ is bounded. Rounding $N^*$ up or down to the nearest integer will only introduce a multiplicative constant no larger than $C 2^\al$, so this yields
\begin{align}\label{e:alambdaKfzetaBound1}
	\abs{\frac{(a_{\la}K) * f(x)}{\eta(x)}}
		\le C \pr{1 + \norm{f}_{L^\iny}}
			\norm{f/\eta}_{C^{-\al}}^{\frac{e^{-C_0t}}{\al + e^{-C_0t}}}
\end{align}
as long an $N^* \ge -1$.

If $N^* < -1$, then it must be that
\begin{align*}
	\norm{f/\eta}_{C^{-\al}}^{\frac{1}{\al + e^{-C_0t}}}
		\ge 2
\end{align*}
so 
\begin{align*}
	\norm{f/\eta}_{C^{-\al}}^{\frac{e^{-C_0t}}{\al + e^{-C_0t}}}
		\ge 1
		\ge Cg(x)
\end{align*}
for some constant $C > 0$.  Then, from \cref{e:finalest},
\begin{align*}
	\abs{\frac{(a_{\la}K) * f(x)}{\eta(x)}}
		&\le C \pr{1 + \norm{f}_{L^\iny}}
			\pr{2^{\al N} \norm{f/\eta}_{C^{-\al}}
			+ 2^{-Ne^{-C_0t}} \norm{f/\eta}_{C^{-\al}}
				^{\frac{e^{-C_0t}}{\al + e^{-C_0t}}}}.
\end{align*}
Choosing $N = -1$, yields
\begin{align}\label{e:alambdaKfzetaBound2}
	\begin{split}
	\abs{\frac{(a_{h(x)}K) * f(x)}{\eta(x)}}
		&\le C \pr{1 + \norm{f}_{L^\iny}}
			\pr{\norm{f/\eta}_{C^{-\al}}
			+ \norm{f/\eta}_{C^{-\al}}^\frac{e^{-C_0t}}{\al + e^{-C_0t}}} \\
		&= C \pr{1 + \norm{f}_{L^\iny}}
			\Phi_\al \pr{t, \norm{\frac{f}{\eta}}_{C^{-\al}}}.
	\end{split}
\end{align}
Adding the bounds in \cref{e:alambdaKfzetaBound1,e:alambdaKfzetaBound2}, we conclude \cref{e:TheJ1Bound} holds for all values of $\norm{f/\eta}_{C^{-\al}}$.
\end{proof}

\begin{lemma}\label{L:zetaFlowBound}
	Let $\zeta$ be a \preGB
	and let $u \in L^\iny(0, T; S_1)$ with $X$ its associated
	flow map. Then for all $t \in [0, T]$, $x \in \R^2$,
	\begin{align}\label{e:zetaFlowBoundEq}
		C(T) \zeta(X^{-1}(t, x))
			\le \zeta(x)
			\le C'(T) \zeta(X^{-1}(t, x))
	\end{align}
	for constants, $0 < C(T) < C'(T)$.
\end{lemma}
\begin{proof}
	We rewrite \cref{e:zetaFlowBoundEq} as
	\begin{align*}
		C(T) \zeta(x)
			\le \zeta(X(t, x))
			\le C'(T) \zeta(x)
	\end{align*}
	for all $x \in \R^2$, a bound that we see easily follows from
	\cref{L:preGB,L:FlowBounds}.
\end{proof}

\begin{lemma}\label{L:ALPBound}
	For all $j \ge -1$ and all $\la > 0$,
	\begin{align*}
		\norm{\grad \Delta_j ((a_\la K) * f)}_{L^\iny}
			\le C \norm{f}_{L^\iny}.
	\end{align*}
\end{lemma}
\begin{proof}
	Write
	\begin{align*}
		(a_\la K) * f(z)
			= \pr{K * (a_\la(\cdot - z) f)}(z).
	\end{align*}
	For fixed $z$, define the divergence-free vector field,
	\begin{align*}
		b_z(w)
			:= \pr{K * (a_\la(\cdot - z) f)}(w).
	\end{align*}
	We know (for instance, from Lemma 4.2 of \cite{CK2006}) that
	$\norm{\grad \Delta_j b_z}_{L^\iny}
			\le C \norm{\Delta_j \curl b_z}_{L^\iny}$. Thus,
	\begin{align*}
		\norm{\grad \Delta_j ((a_\la K) * f)}_{L^\iny}
			&= \norm{\grad \Delta_j b_z}_{L^\iny}
			\le C \norm{\Delta_j \curl b_z}_{L^\iny}
			= C \norm{\Delta_j(a_\la(\cdot - z) f)}_{L^\iny} \\
			&\le C \norm{a_\la(\cdot - z) f}_{L^\iny}
			\le C \norm{f}_{L^\iny}.
	\end{align*}
\end{proof}

\begin{lemma}\label{L:CommutatorEstimate}
	With the assumptions as in \cref{L:ForJ1Bound}, for all
	$x, y \in \R^2$, we have
	\begin{align*}
		&\abs{\Delta_j f(x - y)
				- \eta(x - y) (\Delta_j (f/\eta))(x - y)}
			\le C \norm{f}_{L^\iny}
				\norm{\grad \zeta}_{L^\iny} 2^{-je^{-C_0t}}.
	\end{align*}
\end{lemma}
\begin{proof}
First observe that for any $x$, $y \in \R^2$ 
\begin{equation*}
\begin{split}
&\abs{\eta(x) - \eta(y)} = \frac{\abs{\zeta(X_1^{-1}(t,x)) - \zeta(X_1^{-1}(t,y))}}{\abs{X_1^{-1}(t,x) - X_1^{-1}(t,y)}} \abs{X_1^{-1}(t,x) - X_1^{-1}(t,y)}
\leq  \|\nabla \zeta \|_{L^{\infty}} \chi_t(\abs{x-y})
\end{split}
\end{equation*}
by \cref{L:FlowUpperSpatialMOC}.
	We then have,
	\begin{align*}
		&\abs{\Delta_j f(x - y)
				- \eta(x - y) (\Delta_j (f/\eta))(x - y)} \\
			&\qquad
			= \abs{\int_{\R^2}
				\pr{\varphi_j(z) (\eta f/\eta)(x - y - z)
					- \eta(x - y) \varphi_j(z) (f/\eta)(x - y - z)} \, dz} \\
			&\qquad
			\le \int_{\R^2}
				\abs{\varphi_j(z) (f/\eta)(x - y - z)}
					\abs{\eta(x - y - z) - \eta(x - y)} \, dz \\
			&\qquad
			\le \norm{\grad \zeta}_{L^\iny}
				\int_{\R^2}
				\chi_t(\abs{z}) \abs{\varphi_j(z)} \abs{(f/\eta)(x - y - z)} \, dz
			\le C \norm{\grad \zeta}_{L^\iny} \norm{f}_{L^\iny}2^{-je^{-C_0t}}.
	\end{align*}
	We used here that, because $\varphi$ is Schwartz-class,
	by virtue of \cref{e:chitBound}, we have
	\begin{align*}
		\int_{\R^2} &\abs{\varphi_j(z)} \chi_t(\abs{z})\, dz
			= 2^{2j} \int_{\R^2} \abs{\varphi(2^j z)} \chi_t(\abs{z}) \, dz
				= \int_{\R^2} \abs{\varphi(w)} \chi_t(2^{-j} \abs{w}) \, dw \\
			&\le 2^{-je^{-C_0t}}
				\int_{\R^2} \abs{\varphi(w)} \chi_t(\abs{w}) \, dw
			\le C 2^{-je^{-C_0t}}.
	\end{align*}
\end{proof}

\Ignore{  

\cref{P:f0XBound} extends to log-Lipschitz vector fields similar bounds for Lipschitz vector fields, which have been obtained in great generality for Besov space norms of various types---see Section 3.2 of \cite{BahouriCheminDanchin2011}, for instance. (The bounds for Lipschitz vector fields apply for all time and have no loss of regularity, just an increase in norms.) The version we present here is the log-Lipschitz analog of the special case used by Chemin \cite{C1991,Chemin1993Persistance} in proving the propagation of the regularity of a vortex patch boundary: in our notation, $\norm{f \circ X^{-1}(t)}_{C^{-\al}} \le C e^{C' t} \norm{f}_{C^{-\al}}$ for $\al \in (0, 1)$, where $C'$ is proportional to the Lipschitz constant of the underlying vector field. (A relatively simple proof of this bound appears in Proposition A.3 of \cite{BK2015}.) Observe that, in light of \cref{R:aT}, it is important for us to obtain the dependence of the constant on $\delta$ and $\al$. }  

The following proposition is a simplified version of Theorem 3.28 of \cite{BahouriCheminDanchin2011}. 
\begin{prop}\label{P:f0XBound}
Let $u\in L^{\infty}(0,T; S_1)$.  Assume $f\in C([0,T];L^{\infty})$ solves the transport equation
\begin{equation*}
\begin{split}
&\partial_t f + u\cdot\nabla f = 0\\
&f|_{t=0} = f^0.
\end{split}
\end{equation*}
For fixed $\delta\in(-1,0)$, there exists a constant $C=C(\delta)$ such that for any $\beta>C$, if
\begin{equation}\label{2}
\delta_t = \delta - \beta\int_0^t \| u(s) \|_{LL} \, ds, 
\end{equation}  
and if $T^*$ satisfies $\delta_{T^*} \geq -1$, then
\begin{equation*}
\sup_{t\in[0,T^*]} \| f(t) \|_{C^{\delta_t}} \leq \frac{\beta}{\beta-C} \| f_0 \|_{C^{\delta}}.
\end{equation*}
\end{prop}

\begin{remark}
Proposition \ref{P:f0XBound} is proved in greater generality in \cite{BahouriCheminDanchin2011}.  The authors assume, for example, that $f$ belongs to the appropriate negative \Holder space.  Here, we apply Proposition \ref{P:f0XBound} with $f=\frac{\bar{\omega}^0}{\zeta}\circ X_1^{-1}$, which clearly belongs to $C([0,T];L^{\infty})$ and therefore to all negative \Holder spaces.  In addition, the authors integrate a quantity in (\ref{2}) which differs from $\| u \|_{LL}$ but is bounded above and below by $C\| u \|_{LL}$ for a constant $C>0$.
\end{remark}



\begin{lemma}\label{L:HolderInterpolation}
	For any $u \in S_1(\R^2)$ and $r \in (0, 1)$,
	\begin{align*}
		\norm{u}_{C^r}
			\le
				C \pr{\norm{u}_{L^\iny}
					+
				\norm{u}_{L^\iny}^{1 - r}
				\norm{u}_{S_1}^{r}}.
	\end{align*}	
\end{lemma}
\begin{proof}
We apply \cref{D:LPHolder} and write
\begin{align*}
	\norm{u}_{C_r}
		&= \sup_{j\geq -1} 2^{jr} \norm{\Delta_j u}_{L^\iny} \\
		&\le C\norm{u}_{L^\iny} + \sup_{j\geq 0}
			2^{jr} \norm{\Delta_j u}^{1-r}_{L^\iny}
			\norm{\Delta_j u}^{r}_{L^\iny}\\
		&\le C\norm{u}_{L^\iny}
			+ C \norm{ u}^{1-r}_{L^\iny} \sup_{j\geq 0}
			2^{jr} 2^{-jr}\norm{\Delta_j \grad u}^{r}_{L^\iny},
\end{align*}
where we used Bernstein's Lemma to get the second inequality.
But, using Lemma 4.2 of \cite{CK2006},
\begin{align*}
	\norm{\Delta_j \grad u}_{L^\iny}
		\le C \norm{\Delta_j \curl u}_{L^\iny}
		\le C \norm{\curl u}_{L^\iny}
		\le C \norm{u}_{S^1},
\end{align*}
which yields the result.
\end{proof}


\appendix
%
%
\section{Examples of growth bounds}\label{S:Corollary}

\noindent
\begin{proof}[Proof of \cref{C:MainResult}]
We consider $h(r) = h_1(r) = (1 + r)^\al$ for some $\al \in [0, 1/2)$. Clearly, $h$ and so $h^2$ are increasing and both $h$ and $h^2$ are concave and infinitely differentiable on $(0, \iny)$, and $h'(0) = \al < \iny$. This gives $(i)$ of \cref{D:GrowthBound} for $h$ and $h^2$. Then for $n = 1, 2$,
\begin{align*}
	\int_1^\iny \frac{h^n(s)}{s^2} \, ds
		&= \int_1^\iny \frac{(1 + s)^{n \al}}{s^2} \,ds
		\le 2^{n \al} \int_1^\iny s^{n \al - 2} \,ds
		= \frac{2^{n \al}}{1 - n \al}
		< \iny,
\end{align*}
giving $(ii)$ of \cref{D:GrowthBound} for $h$ and $h^2$. It follows that $h$ is a \GBU. (An additional calculation, which we suppress since we do not strictly need it, shows that $E(r) \le \mu(r) := C(\al) (1 + r^\al) r$, improving the coarse bound of \cref{L:EProp}.)
\Ignore{ 
We now treat $(iv)$ of \cref{D:GrowthBound}. For $r \ge 1$, arguing as above, we have
\begin{align*}
	H(r)
		&= H[h](r)
		\le 2^\al \int_r^\iny s^{\al - 2} \,ds
		= \frac{2^\al}{1 - \al} r^{\al - 1}
		= C_\al r^{\al - 1},
\end{align*}
where $C_\al := 2^\al (1 - \al)^{-1}$.
Thus, for $r \ge 1$,
\begin{align*}
	E(r)
		&\le 2 (1 + r H(r^{\frac{1}{2}})^2) r
		\le 2 (1 + C_\al^2 r^\al) r
		\le C (1 + r^\al) r,
\end{align*}
where in the first inequality we used $(A + B)^2 \le 2(A^2 + B^2)$.

For $r < 1$, on the other hand, we have
\begin{align*}
	H(r)
		&= H(1) + \int_r^1 \frac{(1 + s)^\al}{s^2} \,ds
		\le C_\al + 2^\al \int_r^1 \frac{ds}{s^2}
		= \frac{2^\al}{1 - \al} + \frac{2^\al}{r} - 2^\al \\
		&\le \frac{\al 2^\al}{1 - \al} + \frac{2^\al}{r}
		\le C \pr{1 + \frac{1}{r}}.
\end{align*}
So for $r < 1$,
\begin{align*}
	E(r)
		&\le 2(1 + r H(r^{\frac{1}{2}})^2) r
		\le 2 \pr{1 + C r \pr{1 + \frac{1}{r^{\frac{1}{2}}} }^2} r
		\le 2 \pr{1 + C r \pr{1 + \frac{1}{r}}} r \\
		&\le C (1 + r) r
		\le C(1 + r^\al) r,
\end{align*}
since $r^\al > r$ for $r \in (0, 1)$. Therefore, we can set, for all $r \ge 0$,
\begin{align}\label{e:omu1Bound}
	\mu(r)
		:= C_0(1 + r^\al) r
\end{align}
for a constant $C_0 = C(\al)$ and have $E \le \mu$. Such a $\mu$ clearly satisfies $(iv)$ of \cref{D:GrowthBound}.
} 

Now assume that $h(r) = h_2(r) = \log^{\frac{1}{4}}(e + r)$. Then $h^2(r) = \log^{\frac{1}{2}}(e + r)$. Then $h^2$ is infinitely differentiable on $(0, 1)$, increasing, and concave and hence so also is $h = \sqrt{h^2}$, being a composition of increasing concave functions infinitely differentiable on $(0, 1)$. Also,
\begin{align*}
	(h^2)'(0)
		&= \frac{1}{2} (x + e)^{-1} \log^{-\frac{1}{2}}(x + e)|_{x = 0}
		= \frac{1}{2e} < \iny.
\end{align*}

This gives $(i)$ of \cref{D:GrowthBound}.

Noting that $h_2(r) \le h_1(r) = (1 + r)^\al$ for any $\al > 0$, $(ii)$ and $(iii)$ of \cref{D:GrowthBound} follows for $h$ and $h^2$ from our result for $h_1$. Hence, $h = h_2$ is a \GBU.

To obtain \cref{e:omuOsgoodAtInfinity}, we make the change of variables, $w = 1/s$, giving
\begin{align*}
	H(r)
		= \int_0^{\frac{1}{r}} h^2(1/w) \, dw
		:= \int_0^{\frac{1}{r}} \log^{\frac{1}{2}} \pr{e + \frac{1}{w}} \, dw.
\end{align*}

Now, if $w \le 1/e$ then
\begin{align*}
	\log^{\frac{1}{2}} \pr{e + \frac{1}{w}}
		\le \log^{\frac{1}{2}} \pr{\frac{2}{w}}
\end{align*}
so for $r \ge e$,
\begin{align*}
	H(r)
		&\le \int_0^{\frac{1}{r}} \log^{\frac{1}{2}} \pr{\frac{2}{w}} \, dw
		= \lim_{a \to 0^+}
			\brac{w \sqrt{\log \pr{\frac{2}{w}}}
			- \sqrt{\pi} \erf{} \sqrt{\log \pr{\frac{2}{w}}}}_a^{\frac{1}{r}} \\
		&= \frac{1}{r} \sqrt{\log \pr{2 r}}
			- \sqrt{\pi} \erf \sqrt{\log \pr{2 r}} + \sqrt{\pi}
		= \frac{1}{r} \sqrt{\log \pr{2 r}}
			+ \sqrt{\pi} \erfc \sqrt{\log \pr{2 r}},
\end{align*}
where
\begin{align*}
	\erf(r)
		:= \frac{2}{\sqrt{\pi}} \int_0^r e^{-s^2} \, ds, \qquad
		\erfc(r) = 1 - \erf(r).
\end{align*}
From the well-known inequality,
$
	\erfc(r) \le e^{-r^2},
$
it follows that for $r \ge e$,
\begin{align*}
	H(r)
		\le \frac{1}{r} \sqrt{\log \pr{2 r}}
			+ \frac{\sqrt{\pi}}{2r}.
\end{align*}

If $w > 1/e$ then
\begin{align*}
	\log^{\frac{1}{2}} \pr{e + \frac{1}{w}}
		\le \log^{\frac{1}{2}} \pr{2 e}
\end{align*}
so that for $r < e$,
\begin{align*}
	H(r)
		&= \int_0^{\frac{1}{e}} h^2(1/w) \, dw
			+ \int_{\frac{1}{e}}^{\frac{1}{r}} h^2(1/w) \, dw
		\le H(e)
			+ \int_{\frac{1}{e}}^{\frac{1}{r}} \log^{\frac{1}{2}}
					\pr{2 e} \, dw \\
		&\le \frac{1}{e} \sqrt{\log \pr{2 e}}
			+ \frac{\sqrt{\pi}}{2 e}
			+ \log^{\frac{1}{2}} \pr{2 e} \frac{e - r}{e r}
		= \frac{\sqrt{\pi}}{2 e}
			+ \frac{\log^{\frac{1}{2}} \pr{2 e}}{r}
		\le 2 \frac{\log^{\frac{1}{2}} \pr{2 e}}{r}.
\end{align*}
In the final inequality, we used that
\begin{align*}
	\frac{\sqrt{\pi}}{2 e}
		< 1
		< \frac{\log^{\frac{1}{2}} \pr{2 e}}{r}.
\end{align*}

We see, then, that an inequality that works for all $r > 0$ is
\begin{align*}
	H(r)
		\le 2 \frac{\sqrt{\log \pr{2 e + 2 r}}}{r}.
\end{align*}
This bound on $H$ gives
\begin{align*}
	E(r)
		&\le 2(1 + r H(r^{\frac{1}{2}})^2) r
		\le 2(1 + 4 \log(2 e + 2 r^{\frac{1}{2}})) r \\
		&\le C \pr{1 + \log (e + r)} r
		=: \mu(r).
\end{align*}
Then $\mu(0) = 0$, $\mu$ is continuous, and $\mu$ is convex, since $\mu''(r) = C (r + 2e)/(e + r)^2 > 0$.

Finally,
\begin{align*}
	\int_1^\iny &\frac{dr}{\pr{1 + \log (e + r)} r}
		\ge \frac{1}{2} \int_1^\iny \frac{dr}{\pr{\log (e + r)} r}
		\ge \frac{1}{2} \int_{e + 1}^\iny \frac{dr}{r \log r} \\
		&= \frac{1}{2} \int_{\log(e + 1)}^\iny \frac{d x}{x}
		= \iny,
\end{align*}
where we made the change of variables, $x = \log r$. This shows that \cref{e:omuOsgoodAtInfinity} holds.
\end{proof}

\Ignore{ 
\begin{remark}
If we define $\nu(r) = r^2 \mu(1/r)$ then condition (3) of \cref{T:Existence} becomes
\begin{align*}
	\int_0^1 \frac{dr}{\nu(r)} = \iny.
\end{align*}
That is, $\mu$ satisfies condition (3) of \cref{T:Existence} if and only if $\nu$ is Osgood. In the context of \cref{C:MainResult}, where $\mu(r) = 4 \pr{1 + \log (e + r)} r$, we have
$
	\nu(r)
		= 4 r (1 + \log(e + r^{-1})),
$
which is a log-Lipschitz modulus of continuity.
\end{remark}
} 


%
%
\section*{Acknowledgements}
\noindent This material is based upon work supported by the National Science Foundation under Grant No. DMS-1439786 while the second author was in residence at the Institute for Computational and Experimental Research in Mathematics in Providence, RI, during the Spring 2017 semester.  The first author was supported by Simons Foundation Grant No. 429578.

\bibliographystyle{plain}

\end{document}